\documentclass{article}
\usepackage[utf8]{inputenc}
\usepackage{amsfonts}
\usepackage{amsmath}
\usepackage{graphicx}
\usepackage{amsthm}
\usepackage{hyperref}
\usepackage{amssymb}
\usepackage{tikz}
\usetikzlibrary{shadings}
\usepackage{pgfplots}

\usetikzlibrary{calc}
\usetikzlibrary{positioning, arrows.meta}
\usetikzlibrary{mindmap}
\usepackage{natbib}
\usepackage{outlines}
\usepackage{algorithm}
\usepackage{algorithmic}
\usepackage{comment}
\usepackage{subcaption}
\usepackage{array} 

\newtheorem{thm}{Theorem}
\newtheorem{lem}{Lemma}
\newtheorem{defn}{Definition}
\newtheorem{cor}{Corollary}
\newtheorem{rem}{Remark}
\newtheorem{prop}{Proposition}

\newtheorem{exa}{Example}

\title{On the probability of linear separability through intrinsic volumes}
\author{Felix Kuchelmeister}

\DeclareMathOperator*{\argmin}{arg\,min}

\DeclareMathOperator*{\spa}{span}

\DeclareMathOperator*{\relint}{relint}

\begin{document}

\maketitle
\newpage
\paragraph{Abstract}
A dataset with two labels is linearly separable if it can be split into its two classes with a hyperplane.
This inflicts a curse on some statistical tools (such as logistic regression) but forms a blessing for others (e.g. support vector machines).
Recently, the following question has regained interest:
What is the probability that the data are linearly separable?

We provide a formula for the probability of linear separability for Gaussian features and labels depending only on one marginal of the features (as in generalized linear models).
In this setting, we derive an upper bound that complements the recent result by \cite*{hayakawa2021estimating}, and a sharp upper bound for sign-flip noise.

To prove our results, we exploit that this probability can be expressed as a sum of the intrinsic volumes of a polyhedral cone of the form $\spa\{v\}\oplus[0,\infty)^n$, as shown in \citep{candes2020phase}.
After providing the inequality description for this cone, and an algorithm to project onto it, we calculate its intrinsic volumes.
In doing so, we encounter Youden's demon problem, for which we provide a formula following \cite{kabluchko2020absorption}. 
The key insight of this work is the following:
The number of correctly labeled observations in the data affects the structure of this polyhedral cone, allowing the translation of insights from geometry into statistics.

\newpage
\tableofcontents
\newpage

\section{Introduction}
What is the probability that a dataset with $n$ independent observations and binary labels $(x_j,y_j)$ in $\mathbb{R}^p\times\{-1,+1\}$ can be separated into its classes?
Statisticians and mathematicians have faced this problem in several forms.
Despite the problem's long history dating back to the 1960s \citep{cover1965geometrical}, there have been breakthrough results in recent years, in particular by \cite*{wagner2001continuous,amelunxen2014living,candes2020phase,hayakawa2021estimating} and indirectly, \cite{kabluchko2020absorption}.
Building on these recent results, we provide several results about the probability of linear separability.

The main results of this paper are Theorems \ref{thm_prob_separation} and Theorem \ref{thm_bound_dimension}.
Theorem \ref{thm_prob_separation} gives a formula for the probability of linear separability in generalized linear models (GLM).
Suppose that $x_j\sim\mathcal{N}(0,I_p)$ and the distribution of $y_j$ only depends on $x_j$ through $x_j^T\beta^*$ for some $\beta^*\in \mathbb{R}^p\setminus\{0\}$.
Let $\delta:=\mathbb{P}[yx^T\beta^*>0]$, define $v_j:=y_jx_j^T\beta^*/\|\beta^*\|_2$ and $P:=\spa\{v\}\oplus[0,\infty)^n$.
Let $A_N$ be the event that the first $N$ elements of $v$ are positive, and the rest negative.
By Theorem \ref{thm_prob_separation}, the probability of linear separability without intercept is:
\[
S(p,n)=
\delta^n+(1-\delta)^n+2\sum_{N=1}^{n-1}\sum_{j\geq 1\text{ odd}}{n\choose N}\delta^N(1-\delta)^{n-N}\nu_{n-p+1+j}(P|A_N),
\]
where $\nu_{k}(P|A_N)$ denotes the expected $k$-th intrinsic volume of $P$ given $A_N$.
The biggest technical contribution of this work is that we provide a formula for these intrinsic volumes $\nu_{k}(P|A_N)$.

Theorem \ref{thm_bound_dimension} gives an upper bound for the probability of linear separability.
Suppose that $x_j\sim\mathcal{N}(0,I_p)$, and that for some $\beta^*\in S^{p-1}$, the labels $y_j$ only depend on $x_j$ through $x_j^T\beta^*$.
Suppose that for some $t\geq0$, 
\[
\frac{2(p-1)}{n}+\left(4+\frac{1}{\sqrt{2}}\right)\sqrt{\frac{t}{n}}\leq \min\{\delta,1-\delta\}.
\]
Then, by Theorem \ref{thm_bound_dimension} the probability of linear separability without intercept is at most:
\[
S(n,p)\leq 3\exp\left(-t\right).
\]

We proceed as follows:
First, we motivate the study of this problem in Section \ref{sec_int_motivation}.
Then, we define the statistical models we will work with in Section \ref{sec_int_models}.
We review the existing literature including the recent contributions in Section \ref{sec_int_existing}.
In Section \ref{sec_int_cont}, we summarize our contributions to this literature.
Throughout this work, observations $(x_j,y_j)$ are taken as i.i.d.
Moreover, we take $p<n$, since if $p\geq n$ and the matrix with entries $x_{j,k}$ has full rank, the probability of linear separability is $1$.

\subsection{Motivation}\label{sec_int_motivation}

In generalized linear models (GLM) \cite{nelder1972generalized} for binary data, a parameter $\beta^*\in\mathbb{R}^p$ is estimated using an empirical risk minimizer:
\begin{equation}\label{eq_glm}
\hat{\beta}:=\argmin_{\beta\in\mathbb{R}^p}\frac{1}{n}\sum_{j=1}^n l(y_jx_j^T\beta),
\end{equation}
where $l$ is some loss function.
For many popular choices of $l$, the estimator $\hat{\beta}$ is well-defined if and only if the data are not linearly separable.
For example, this is the case for logistic regression \citep{silvapulle1981existence}.

This provides a problem for the study of such estimators, as they are not well-defined with positive probability.
One way to work around this issue is to study the estimator asymptotically, where the probability is known to exhibit a sharp phase transition and to be either $0$ or $1$
\citep*{sur2019likelihood,sur2019modern,zhao2022asymptotic}.
However, this is not the case in finite samples.
To study finite sample GLMs, one can constrain the domain of \eqref{eq_glm} to a ball of an arbitrarily large radius to make the estimator well-defined \citep{kuchelmeister2023finite}.
So, understanding when the data are not linearly separable with high probability allows us to study \eqref{eq_glm} in finite samples and 
in high-dimensional asymptotics.

If we extend the domain in \eqref{eq_glm} to allow for vectors of infinite length, the problem becomes well posed if the data are linearly separable.
The solution then coincides with all vectors of infinite length, that linearly separate the data.
Therefore, it coincides with the set of minimizers of the 0/1 loss $\sum_{j=1}^n1\{y_jx_j\beta<0\}$ of infinite length.
A penalized version of \eqref{eq_glm}, say with $+\lambda\|\beta\|_2^2$ for a very small $\lambda$, behaves as the max-margin estimator \citep*{rosset2003margin}.
Minimizers of the 0/1-loss and margin maximizing estimators have been studied extensively, e.g. in \citep*{massart2006risk,balcan2013active,bousquet2020proper}.
So, knowledge in which regimes the data are linearly separable with high probability gives insights into the behavior of the direction of $\hat{\beta}$.

Some estimators benefit from the circumstance that the data are linearly separable, or entirely rely on it.
For example, the hard-margin support vector machine (SVM) \citep{cortes1995support} is defined, if and only if the data are separable, opposite to the estimator in \eqref{eq_glm}.
Consequentially, such estimators are often analyzed on datasets that are inherently linearly separable (see e.g. \citep*{bousquet2020proper}), or which are separable with high probability \citep*{montanari2019generalization}.

Linear separability of the data may also bring computational benefits.
For example, it is known that finding a minimizer of the 0/1-loss is an NP-complete problem \citep*{johnson1978densest,feldman2012agnostic}.
However, if the data are linearly separable, a solution can be found efficiently - for example, using linear programming.

\subsection{Definitions and models}\label{sec_int_models}
Here, we provide some core definitions and introduce the statistical models we will work with.
We first define the probability of linear separability, distinguishing whether an intercept is used in Section \ref{sec_intro_defn}.
Then, we introduce class imbalance in Section \ref{sec_class_imbalance}, which alone can make the data separable with high probability.
Finally, Section \ref{sec_models} introduces our assumptions on the labels $y$.

\subsubsection{Linear separation with or without intercept}\label{sec_intro_defn}
We say that a dataset with observations $(x_1,y_1),\ldots,(x_n,y_n)$ in $\mathbb{R}^n\times\{-1,+1\}$ is linearly separable, if there exists a $\beta\in \mathbb{R}^p$, such that for all $j\in\{1,\ldots,n\}$,
\begin{equation}\label{eq_intro_sign}
sign(x_j^T\beta)=y_j.
\end{equation}
Here, we define $sign(z)$ as $1$ if $z>0$, $-1$ if $z<0$ and $0$ if $z=0$.
Equivalently, there exists a $\beta\in\mathbb{R}^p$ such that for all $j\in\{1,\ldots,n\}$, it holds that $y_jx_j^T\beta>0$.
We remark that variants of this definition exist in the literature (see Appendix \ref{app_defn}). 
For our purposes, we do not need to distinguish between them, as they are almost surely equivalent see Remark \ref{rem_equivalence}.
We arrive at the following definition:

\begin{defn}
The probability of linear separability without intercept is:
\begin{equation}\label{eq_defn_separate}
S(n,p):=\mathbb{P}\left[\bigcup_{\beta\in \mathbb{R}^p}\bigcap_{j=1}^n y_jx_j^T\beta>0\right].
\end{equation}
The probability of linear separability with intercept is:
\begin{equation}\label{eq_defn_separate_0}
S_0(n,p):=\mathbb{P}\left[\bigcup_{(\beta_0,\beta)\in \mathbb{R}^{p+1}}\bigcap_{j=1}^n y_j(\beta_0+x_j^T\beta)>0\right].
\end{equation}
\end{defn}

Using the transformed covariates $\tilde{x}_j:=(1,x_j^T)^T$, some results about $S(n,p+1)$ can be translated to results about $S_0(n,p)$.
However, there are other cases, where the intercept makes an enormous difference, see Remark \ref{rem_imbalance}.
From now on we distinguish between the probability of linear separability with/without intercept.

\subsubsection{Class imbalance}\label{sec_class_imbalance}
One of the parameters that may influence the probability of linear separability is class imbalance.
We say that there is \textit{class balance}, if:
\begin{equation}\label{eq_defn_balanced}
b:=\mathbb{P}[y=1]=\mathbb{P}[y=-1]=\frac{1}{2}.
\end{equation}
If instead $b\neq 1/2$, we say that there is \textit{class imbalance}.
In Remark \ref{rem_imbalance} below we discuss the importance of class imbalance to the probability of linear separability.

\begin{rem}\label{rem_imbalance}
Class imbalance and the inclusion/exclusion of an intercept can affect the probability of linear separability.
While this follows from Proposition \ref{prop_prob_intercept}, we give a simple example to illustrate this point here.
Suppose that $x_j\sim\mathcal{N}(0,I)$ and $y_j=1$ almost surely, in the notation of \eqref{eq_defn_balanced}, $b=1$.
Then, it is easy to check that the probability of linear separability \textit{with} intercept is 1.
At the same time, Cover's Theorem \eqref{eq_cover} shows us that the probability \textit{without} intercept is $2^{n-1}\sum_{k=0}^{p-1}{n-1\choose k}$.
Compare this to the case where the $y_j$ are independent of the $x_j$ and there is class balance, $b=1/2$.
Then, by \eqref{eq_cover} the probability with intercept would be $2^{n-1}\sum_{k=0}^{p}{n-1\choose k}$, while the probability without intercept would remain unchanged.
So, depending on the underlying model, class imbalance and the inclusion/exclusion of an intercept affect the probability of linear separability to a varying degree.
\end{rem}

\subsubsection{Statistical models}\label{sec_models}
Here, we present several statistical models, for which we present results about $S(n,p)$ and $S_0(n,p)$.
Throughout, we have $n$ i.i.d. paired random variables $x_j,y_j$ taking values in $\mathbb{R}^p\times\{-1,+1\}$, where $x_j$ are the \textit{features} (also called \textit{covariates}) and $y_j\in\{-1,+1\}$ are the labels.
Here, we discuss how the pairs of features $x_j$ and labels $y_j$ are related.
We introduce the models with intercept, the models without intercept are defined analogously.

\subsubsection*{Halfspace learning model}
From a statistical learning point of view, one can see $sign(\beta_0+x^T\beta)$ as the prediction of the vector $(\beta_0,\beta)$ at the point $x$.
So, $(\beta_0,\beta)$ predicts the label $y$ correctly, if $y(\beta_0+x^T\beta)>0$.
In this model, we assume that there is a unique vector $(\beta_0^*,\beta^*)\in\mathbb{R}^{p+1}$, which makes the fewest mistakes in expectation.
Letting $S^{p-1}:=\{\beta\in\mathbb{R}^p:\|\beta\|_2=1\}$, we define:\footnote{
Definition \eqref{eq_oracle} can exhibit degenerate behavior if there exists a $(\beta_0,\beta)$ such that $x^T\beta+\beta_0=0$ with nonzero probability.
This only affects us in the discussion of \eqref{eq_satoshi_upper}.}
\begin{equation}\label{eq_oracle}
(\beta_0^*,\beta^*):=\argmin_{(\beta_0,\beta)\in \mathbb{R}\times S^{p-1}}\mathbb{P}[y(\beta_0+x^T\beta)<0]
\end{equation}
In other words, $(\beta_0^*,\beta^*)$ is the vector that is most likely to predict the label $y$.

The vector $(\beta_0^*,\beta^*)$ has many names, such as \textit{ground truth}, \textit{target} or \textit{oracle}.
In inferential statistics, one often views the labels as generated by $(\beta_0^*,\beta^*)$, at a point $x$.
Then, $sign(\beta^*_0+x^T\beta^*)$ is in some sense the `correct label', and $y$ is viewed as a noisy version of $sign(\beta^*_0+x^T\beta^*)$.
We can then say that the \textit{probability of observing a correct label} is:
\begin{equation}\label{eq_defn_delta}
\delta:=\mathbb{P}[y(\beta^*_0+x^T\beta^*)>0],
\end{equation}
whereas $1-\delta$ is the probability of observing a \textit{`wrong label'}.

\subsubsection*{Generalized linear model}
In generalized linear models \citep{nelder1972generalized}, it is assumed that the labels $y$ only depend on $x$ through the marginal $x^T\beta^*$.\footnote{In generalized linear models, it is classically assumed that the conditional probability $y|x$ is an invertible function of $\beta_0^*+x^T\beta^*$.
We do not use this assumption.}
\begin{equation}\label{eq_dependence_xTbeta}
\mathbb{P}[y=1|x]=\mathbb{P}[y=1|x^T\beta^*].
\end{equation}
In other words, all information in $x$ about $y$ is contained in $x^T\beta^*$.
Typically, the labels in these models can be generated as follows:
\begin{equation}\label{eq_logit}
y:=sign(\beta_0^*+x^T\beta^*+\epsilon),
\end{equation} 
where $\epsilon$ is independent of $x$.
If $\epsilon$ follows a logistic distribution, this gives the logit model, if it follows a Gaussian distribution, the probit model.
Under the assumption \eqref{eq_dependence_xTbeta}, if the features $x$ follow a Gaussian distribution, then $y$ is independent of all $x^T\beta$ for $\beta$ orthogonal to $\beta^*$.
For example, if $\beta^*=e_1$, then:
\begin{equation}\label{eq_gaussian_trick}
y_jx_j=y_jx_{j,1}+y_jx_{j,2:p}\sim y_jx_{j,1}+x_{j,2:p},
\end{equation}
where $\sim$ means equivalence in distribution.
Our analysis will heavily rely on the decomposition \eqref{eq_gaussian_trick}.

\subsubsection*{Sign-flip noise}
Given the `oracle' $(\beta_0^*,\beta^*)$ from \eqref{eq_oracle}, we can define the probability that one observes a `correct label' $y$, i.e. a label that agrees with the oracle.
This probability is constant for all $x$ in the \textit{sign-flip noise} model.
Formally:
\begin{equation}\label{eq_sign_flip}
\mathbb{P}[y=1|x]=\begin{cases}
\delta&\text{if }\beta_0^*+x^T\beta^*>0,\\
1-\delta&\text{if }\beta_0^*+x^T\beta^*\leq 0.
\end{cases}
\end{equation}
This data can be generated as follows.
First, generate a `correct', uncorrupted label $sign(x^T\beta^*)$ at a point $x$, and then `flip' it with probability $1-\delta$.
The condition \eqref{eq_sign_flip} is a special case of the condition \eqref{eq_dependence_xTbeta}.
Conditional on the event that we have a correct observation $y_j=x_j^T\beta^*$, the random variable $y_jx_j^T\beta^*$ has the same distribution as $|x^T_j\beta^*|$.
This simplifies some of our expressions (see Remark \ref{rem_sign_flip}).

Although the sign-flip model is relatively easy to analyze for our purposes, it is quite restrictive.
For example, if the distribution of $x$ is symmetric, then there is class balance in the sign-flip model.

\subsubsection*{Signal-less imbalanced model}
The simplest model we study is the one where $y$ and $x$ are independent:
\begin{equation}\label{eq_intercept}
\mathbb{P}[y=1|x]=\mathbb{P}[y=1]=:b
\end{equation}
This model allows us to isolate how class imbalance affects the probability of linear separability.
We have already seen in Remark \ref{rem_imbalance}, that the intercept plays a crucial role here.

\subsection{Existing results on the linear separability problem}\label{sec_int_existing}

The literature on the probability of linear separability dates from the earliest discoveries in the 1960s to the recent breakthroughs in the 2020s.
Here, we present several of these existing results.
First, we state the classical identity attributed to Cover, the universal lower and upper bounds, and related finite sample results.
We mention a line of research on high-dimensional asymptotics, which shows that the probability asymptotically exhibits a sharp phase transition.
Finally, we present the argument that relates the probability of linear separability to the intrinsic volumes through the kinematic formula.

Although we attempt to mention many relevant results, this is not a complete literature review.
In particular, the interested reader will find more approximations in \cite*{hayakawa2021estimating}.

\subsubsection*{Equivalence to random convex hull problem}
The linear separability problem has a twin in the convex geometry literature:
What is the probability that the convex hull of $n$ i.i.d. points in $\mathbb{R}^p$ contains the origin?
By the hyperplane separation theorem, the convex hull $Conv$ of $n$ points does not contain the origin if and only if the points lie in a homogeneous open half-space.
Consequentially:
\[
\mathbb{P}\left[\bigcup_{\beta\in \mathbb{R}^p}\bigcap_{j=1}^n y_jx_j^T\beta>0\right]
=\mathbb{P}\left[0\not\in Conv(y_1x_1,\ldots,y_nx_n)\right].
\]
This equivalence was already implicitly pointed out in the work of \cite{cover1965geometrical}.
Since then, several great results have been achieved in the literature about the random convex hull problem.
In this section, we state results relying on them, applied to our context.

We remark that there are other equivalent definitions of linear separability, which we do not discuss in detail here.
For example, the data are linearly separable if the convex hull of the points in class $1$ does not intersect the convex hull of the points in class $-1$.\footnote{Without intercept, one could use the positive hull instead of the convex hulls. This was the approach by \cite{silvapulle1981existence}.}
This definition has a natural extension for multiclass classification (see \cite{de2019stochastic}).

\subsubsection{Identities and finite sample bounds}\label{sec_identities_bounds}
\subsubsection*{Identities}
There are not many known identities for the probability of linear separability.
A famous exception is the following:
If the i.i.d. random variables $y_jx_j$ have a symmetric distribution and are almost surely linearly independent, then the probability of linear separability without intercept is:
\begin{equation}\label{eq_cover}
S(n,p)=\frac{1}{2^{n-1}}\sum_{k=0}^{p-1}{n-1\choose k}
\end{equation}
We can apply this result to $y_j(x_j^T,1)^T$ to obtain the analogous identity for the probability of linear separability with intercept.
The result is related to the fact that there are $2\sum_{k=0}^{p-1}{n-1\choose k}$ ways to partition $n$ linearly independent points with homogeneous hyperplanes, a result due to Ludwig Schläfli.
In this form, the result is due to \citep{wendel1962problem}, although several authors independently proved an equivalent result in the early 1960s, see \cite{cover1965geometrical}.
Some authors refer to \eqref{eq_cover} as \textit{Cover's Theorem}.

Besides \eqref{eq_cover}, there are few known identities for the probability of linear separability.
An exception is the identity below by \cite{bruckstein1985monotonicity}.
They exploit that in the one-dimensional case, the problem has a much simpler structure than in the higher-dimensional case.

Finally, an identity for the case $y_jx_j\sim\mathcal{N}(\mu,I_p)$ follows from the work by \cite{kabluchko2020absorption}.

\subsubsection*{Lower bounds}
\cite{bruckstein1985monotonicity} conjectured that the probability of linear separability is always at least as large as in \eqref{eq_cover} (see also \citep{cover1987open}).
They gave the first affirmative result for a special case, where the conditional density $x|y$ of the first class is a translation of the conditional density for the second class.
\cite{wagner2001continuous} confirmed the conjecture for all i.i.d. random variables $y_jx_j$, whose density is absolutely continuous with respect to the Lebesgue measure.
Equality is achieved, if and only if the distribution is balanced about the origin \cite[Corollary 3.7.]{wagner2001continuous}.
This result was recently generalized by \cite*{hayakawa2021estimating}.
The assumption of absolute continuity is weakened to the condition $S(p,p)=1$, the weakest possible assumption for which the inequality holds for all $p\leq n$.

While the universal lower bound is a beautiful and strong result, it does not show the dependency on the signal strength.
For the case where $x_j\sim\mathcal{N}(0,I_p)$ and $y$ follows the logit model (see \eqref{eq_logit}), \cite{chardon2024finite} show that if $n\leq c_1pB/\kappa$ for some $\kappa\geq 1$, then the data are linearly separable with probability at least $1-c_2\exp({-c_3\max\{\kappa\sqrt{p},\kappa^2 p/B^2\}})$, where $c_1,c_2,c_3>0$ are constants and $B:=\max\{\|\beta^*\|_2,e\}$ \cite*[Theorem 2]{chardon2024finite}.

\subsubsection*{Upper bounds}
There are some known upper bounds on the probability of linear separability.
In particular, there is a universal upper bound.
In \cite*[Theorem 14]{hayakawa2021estimating}, it is proved that with i.i.d. observations $(x_i,y_i)$, the probability of linear separability without intercept satisfies:
\begin{equation}\label{eq_satoshi_upper}
\inf_{\beta\in S^{p-1}}\mathbb{P}[yx^T\beta<0]\geq 3\frac{p}{n},\quad\Rightarrow \quad S(n,p)< \frac{1}{2^p}.
\end{equation}
Applying this result to $y(1,x^T)^T$, we obtain the analogous bound for the probability of linear separability with intercept.
In the signal-less imbalanced model \eqref{eq_intercept}, $\alpha(y(x^T,1)^T)=1-b$, assuming $b\geq 1/2$.
Consequentially, \eqref{eq_satoshi_upper} also provides insights into how class balance affects the probability of linear separability.

The infimum over all $\mathbb{P}[yx^T\beta<0]$ is referred to as the \textit{Tukey depth} and has a nice interpretation in the halfspace learning model \eqref{eq_oracle}.
Assuming that $x^T\beta=0$ with probability zero for all $\beta\in S^{p-1}$, we get:
\[
\inf_{\beta\in S^{p-1}}\mathbb{P}[yx^T\beta<0]=\mathbb{P}[yx^T\beta^*<0]=1-\delta.
\]
We defined $\delta$ as the probability of observing a `correct' label $y$, i.e. a label $y$ that agrees with the sign of $x^T\beta^*$, see \eqref{eq_defn_delta}.

While \eqref{eq_satoshi_upper} gives a high-probability result if $p$ is large, there is a result that targets the regime where $p$ is small and $n$ is large.
In \cite[Proposition 2.3.1]{kuchelmeister2023finite}, the probit model is studied, in particular \eqref{eq_logit} with $(x^T,\epsilon)^T\sim\mathcal{N}(0,I_{p+1})$.
They prove that for $t>0$, if $(p\log n+t)/n\lesssim1-\delta\lesssim 1$, then $ S(n,p)\leq 4\exp(-t)$.
Here, $\lesssim$ means smaller than up to a constant factor.

After the initial version of this work was made available \citep{kuchelmeister2024probability}, an independent and related approach was proposed by \cite{chardon2024finite}.
While we state our results in terms of the probability of observing a correct label $\delta$, they formulate the dependency on the signal in terms of the norm of the ground truth $\|\beta^*\|_2$, respectively $B:=\max\{\|\beta^*\|_2,e\}$.
For the case where the probability of observing a wrong label $1-\delta$ is small, it is known that $\|\beta^*\|_2(1-\delta)\sim 1$ for the probit model \cite[Section 3.2]{kuchelmeister2023finite} and $\|\beta^*\|_2^2(1-\delta)\sim 1$ for the sign-flip noise model \citep[Lemma 26]{chardon2024finite} where by $a\sim b$ we mean $a\lesssim b$ and $b\lesssim a$.
For the logit model with Gaussian features, \cite[Theorem 1]{chardon2024finite}  shows that for $t>0$:
\begin{equation}\label{eq_chardon_1}
n\gtrsim B(p+t)\quad\Rightarrow\quad S(n,p)\leq \exp(-t).
\end{equation}
They complement this result with \cite[Theorem 3]{chardon2024finite}, which assumes the logit model but relaxes the assumptions on $x$, and carries extra $\log^4(B)$ factor.
For a general halfspace learning model, under some regularity assumptions on the distribution of $(x,y)$, 
\cite[Theorem 4]{chardon2024finite} shows that for $t>0$:
\begin{equation}\label{eq_chardon_2}
n\gtrsim B\log(B)^4(p+Bt) \quad\Rightarrow\quad S(n,p)\leq \exp(-t).
\end{equation}

\subsubsection{High-dimensional asymptotics}\label{sec_hd}
As both the sample size $n$ and the dimension $p$ tend to infinity with $p/n\rightarrow \kappa$, the probability of linear separability exhibits a phase transition depending on this ratio $\kappa$.
It was first observed in the signal-less balanced model, as in \eqref{eq_cover}.
In this case, the probability of linear separability tends to $1$ if $\kappa>1/2$ and to $0$ if $\kappa<1/2$ \citep{winder1963bounds,cover1965geometrical}.

Decades later, \cite{candes2020phase} provided a phase transition for a generalized linear model \eqref{eq_logit} with logit link function, with Gaussian features $x_j$.
Assuming without loss of generality that $\beta^*/\|\beta^*\|_2=e_1$, they showed:
\begin{equation}\label{eq_candes}
\lim_{\substack{n\rightarrow\infty\\ p/n\rightarrow\kappa} }S_0(n,p)=\begin{cases}
0,& \text{if }\kappa<h^*,\\
1,&\text{if }\kappa>h^*,
\end{cases}
\end{equation}
where, for $g\sim\mathcal{N}(0,1)$ independent of $x$ and $y$,
\begin{equation}\label{eq_candes_h}
h^*=\min_{t_0,t_1}\mathbb{E}\left[\max\left\{t_0y+t_1yx_{1}-g,0\right\}^2\right],
\end{equation}
see \cite[Theorem 2.1]{candes2020phase}.
They show similar results for the probability of linear separability without intercept and the signal-less imbalanced model.
In this model, $(\beta^*_0,\beta^*)$ and $\delta$ (the probability of observing a correct label) are in a one-to-one relationship, hence we could write $f(\delta)=h^*$.
While the expression $h^*$ is somewhat involved, some upper bounds are available, including \eqref{eq_satoshi_upper}.

The result in \cite{candes2020phase} was later extended beyond the logistic link function and the Gaussian assumption by several authors
\citep*{montanari2019generalization, tang2020existence,zhao2022asymptotic,deng2022model}.
For us, the work by \cite{candes2020phase} is particularly interesting because it introduced a tool from geometry to our problem: the kinematic formula.
It forms the starting point for this work.

\subsubsection{The kinematic formula}\label{subsub_kinematic}
Here, we introduce the kinematic formula and explain how it relates to the probability of linear separability.
This will lead us to the problem of calculating intrinsic volumes of polyhedral cones of the form $\spa\{v\}\oplus[0,\infty)^n$, which forms the main technical contribution of this work.

In \cite{candes2020phase}, it was observed that in the case where the features $x_j$ are Gaussian, the problem of linear separability has yet another interpretation.
It can be formulated as the probability that a uniformly rotated linear subspace intersects a polyhedral cone.
Such a connection was first exploited in compressed sensing \citep*{amelunxen2014living}.

Before we state the result, we introduce some notation.
For two sets $A,B\subset\mathbb{R}^n$ we denote their Minkowski sum by $A\oplus B:=\{a+b:a\in A,b\in B\}$.
We denote the set of rotations (i.e. the special orthogonal group, the set of matrices $A\in \mathbb{R}^{n\times n}$ with determinant 1 satisfying $A^TA=I_n$) by $SO(n)$.
There exists a (unique) uniform distribution on $SO(n)$ (the Haar measure, see e.g. \cite[Chapter 13]{schneider2008stochastic}).

\begin{prop}\label{prop_candes}
Suppose that $x_j\sim\mathcal{N}(0,I_p)$, and that for some $\beta^*\in\mathbb{R}^p\setminus\{0\}$, the distribution of $y_j$ only depends on $x_j$ through $x_j^T\beta^*$.
Let $v_j:=y_jx_j^T\beta^*/\|\beta^*\|_2$.
Then, the probability of linear separability without intercept is:
\[
S(n,p)=\mathbb{P}\left[\theta L_{p-1}\cap \spa\{v\}\oplus[0,\infty)^n\neq \{0\}\right],
\]
and in the signal-less imbalanced model \eqref{eq_intercept}, the probability of linear separability with intercept is:
\[
S_0(n,p)=\mathbb{P}\left[\theta L_{p}\cap \spa\{y\}\oplus[0,\infty)^n\neq \{0\}\right],
\]
where $L_k$ is a linear subspace of dimension $k$ in $\mathbb{R}^n$, and $\theta$ is a uniform random rotation in $SO(n)$ independent of $y$ and $v$.
\end{prop}

This result is a refinement of  \cite[Propositions 1 \& 2]{candes2020phase}.
We exploit that there are several (in our case equivalent) definitions of linear separability.
This step leads to a technical tangent, for which we give a detailed discussion in Appendix \ref{app_defn}.
Although we do not prove it, with similar arguments one could show that the probability of linear separability with intercept and signal is:
\[
S_0(n,p)=\mathbb{P}\left[\theta  L_{p-1}\cap \spa\{y\}\oplus\spa\{v\}\oplus[0,\infty)^n\neq \{0\}\right].
\]

\begin{proof}[Proof of Proposition \ref{prop_candes}]
We will only prove the statement without an intercept, the other case is proved analogously.
Without loss of generality, assume that $\|\beta^*\|_2=1$.
Since $x_j$ has a rotationally invariant distribution, we can assume without loss of generality that $\beta^*=e_1$.
Then, using \eqref{eq_gaussian_trick} and Proposition \ref{prop_equivalence} (respectively Remark \ref{rem_equivalence}),
\[
S(n,p)=
\mathbb{P}\left[\bigcup_{\beta\in\mathbb{R}^{p}}\bigcap_{j=1}^n v_j\beta_1+x_{j,2:p}^T\beta_{2:p}>0\right]
\]
\[=\mathbb{P}\left[\spa\{x_{\cdot,j}:j\geq 2\}\cap \spa\{v\}\oplus[0,\infty)^n\neq \{0\}
\right].
\]
Finally, $\spa\{x_{\cdot,j}:j\geq 2\}$ is a random $p-1$ dimensional random linear subspace of $\mathbb{R}^n$.
As the $x_j$ are Gaussian, its distribution is invariant under rotations.
By uniqueness of the Haar measure, it follows that it has the same distribution as $\theta L_{p-1}$.
The proof is complete.
\end{proof}

The identity in Proposition \ref{prop_candes} converts one calculation into another:
We now want to calculate the probability that a random uniform linear subspace of dimension $p-1$ intersects a polyhedral cone, which has a different shape depending on the problem.
This probability can be calculated with the kinematic formula (Theorem \ref{thm_kinematic}), which we introduce next.

Before introducing the kinematic formula, we introduce some definitions from polyhedral geometry (see e.g. \cite{schrijver1998theory}).
A \textit{polyhedron} is a finite intersection of closed half-spaces.
We denote the faces of a polyhedron $P$ by $\mathcal{F}(P)$, and the $k$-faces by $\mathcal{F}_k(P)$.
If it is clear from context, we omit the argument $P$ and simply write $\mathcal{F}$.
A set $C\subset\mathbb{R}^n$ is a \textit{cone}, if for any $\lambda>0$ and $c\in C$, $\lambda c\in C$.
A \textit{polyhedral cone} is a cone that is a polyhedron.
For a set $A\subset \mathbb{R}^n$, we use the notation $\relint A$ for its \textit{relative interior}, i.e. the interior of $A$ with respect to the affine hull of $A$.
In other words, it is the interior of $A$ in the smallest affine subspace that contains $A$. 
The \textit{projection} of $x\in\mathbb{R}^n$ onto a closed convex set $K\subset\mathbb{R}^p$ is defined as the closest point in $K$ to $x$, with respect to the Euclidean norm.
We denote it by $\Pi_K(x)$.
The kinematic formula relates the probability that a cone intersects a random linear subspace with its intrinsic volumes.
These are defined next.

\begin{defn}\label{defn_intrinsic}
Let $C\subset \mathbb{R}^n$ be a polyhedral cone and $g\sim\mathcal{N}(0,1)$.
For any face $F\in\mathcal{F}(C)$, we define:
\[
\nu_F(C):=\mathbb{P}[\Pi_C(g)\in\relint F].
\]
For $k\in\{0,\ldots,n\}$ $k$-th \textit{intrinsic volume} of $C$ is defined as:
\[
\nu_k(C):=\sum_{F\in\mathcal{F}_k(C)}\nu_F(C).
\]
\end{defn}

In other words, the $k$-th intrinsic volume of the polyhedral cone $C$ is the probability that a Gaussian vector when projected onto the cone $C$, lands on any face of dimension $k$.
We are now ready for the kinematic formula.
For a proof see \cite[Theorem 4.3.5]{schneider2022convex}.

\begin{thm}[Kinematic Formula]\label{thm_kinematic}
Let $C\subset\mathbb{R}^n$ be a closed convex cone which is not a subspace.
Let $L_{n-k}$ be a linear subspace of dimension $n-k$ and $\theta$ be a uniform random rotation in $SO(n)$.
Then, 
\[
\mathbb{P}[C\cap \theta L_{n-k}\neq \{0\}]=2\sum_{j\geq 1\text{ odd}}\nu_{k+j}(C).
\]
\end{thm}

So, Theorem \ref{thm_kinematic} together with Proposition \ref{prop_candes} show that we can express the probability of linear separability through the intrinsic volumes of polyhedral cones such as $\spa\{v\}\oplus[0,\infty)^n$.
According to \cite[p. 98]{schneider2022convex} it is in general not possible to compute the intrinsic volumes of polyhedral cones.
Yet, there are some known successful examples, including \citep{amelunxen2012intrinsic,kabluchko2020absorption,godland2022conic}.

Before we proceed, we consider some approximations related to the kinematic formula.
The intrinsic volumes form a probability distribution on $\{0,\ldots,n\}$.
Therefore, we can define the random variable $\nu(C)$, which takes the value $k$ with probability $\nu_k(C)$.
A simple monotonicity argument \cite[Lemma 4.3.3]{schneider2022convex} gives:
\begin{equation}\label{eq_tailbounds}
\mathbb{P}[\nu(C)\geq k+1]\leq 
\mathbb{P}[C\cap \theta L_{n-k}\neq \{0\}]\leq \mathbb{P}[\nu(C)\geq k]
\end{equation}
So, the probability of linear separability behaves as the tail probability of the random variable $\nu(C)$.
Its mean $\delta(C):=\sum_{k=0}^nk\nu_k(C)$ is referred to as the \textit{statistical dimension} of $C$.
It is known that $\delta(C)$ behaves as the squared Gaussian mean-width of $C$, see \cite*[Proposition 10.2]{amelunxen2014living}.
Moreover, the random variable $\nu(C)$ concentrates around its mean, $\delta(C)$.
The following is a Bernstein-type inequality, a consequence of \cite[Corollary 4.10]{mccoy2014steiner}, where a stronger Bennett-type inequality is provided.
It is an improvement over an earlier Bernstein-type inequality, provided in \cite*[Theorem 6.1]{amelunxen2014living}.

\begin{thm}[\cite{mccoy2014steiner}]\label{thm_mccoy}
Let $C\subset\mathbb{R}^n$ be a closed convex cone.
For each $t\geq 0$, the intrinsic volume random variable $\nu(C)$ satisfies:
\[
\mathbb{P}\left[\nu(C)-\delta(C)\geq t\right]
\leq \exp\left(
-\frac{t^2/4}{(\delta(C)+t/3)\wedge(\delta(C^\circ)-t/3)}
\right),
\]
\[
\mathbb{P}\left[\nu(C)-\delta(C)\leq -t\right]
\leq \exp\left(
-\frac{t^2/4}{(\delta(C)-t/3)\wedge(\delta(C^\circ)+t/3)}
\right).
\]
\end{thm}

Above, $\wedge$ denotes the minimum, and $C^\circ$ is the polar of the cone $C$.
For the proof, we refer the reader to \cite[Corollary 4.10]{mccoy2014steiner}.
As $\delta(C),\delta(C^\circ)\leq n$, this gives the upper bound $\exp(-t^2/4n)$ in both cases.
Using Theorem \ref{thm_kinematic}, the approximations \eqref{eq_tailbounds} and Theorem \ref{thm_mccoy} one can obtain upper and lower bound on the probability of linear separability.
For any $t\geq 0$, we get:
\begin{equation}\label{eq_approx_kin_up}
\frac{\delta(C)}{n}\leq 1-\frac{p}{n}-2\sqrt{\frac{t}{n}}\quad
\Rightarrow\quad
\mathbb{P}[C\cap \theta L_{p}\neq \{0\}]\leq \exp\left(-t\right),
\end{equation}
\begin{equation}\label{eq_approx_kin_down}
\frac{\delta(C)}{n}\geq 1-\frac{p-1}{n}+2\sqrt{\frac{t}{n}}\quad
\Rightarrow\quad
\mathbb{P}[C\cap \theta L_{p}\neq \{0\}]\geq 1-\exp(-t).
\end{equation}
This is the `\textit{approximate kinematic formula}' \cite*[Theorem 1]{amelunxen2014living}, which was used in \cite{candes2020phase} (the constants are improved since we use Theorem \ref{thm_mccoy}).
Beware that there is a technicality lurking here:
The conditions in \eqref{eq_approx_kin_up} and \eqref{eq_approx_kin_down} are deterministic.
However, as in our case the polyhedral cone $P:=\spa\{v\}\oplus[0,\infty)^n$ is random, $\delta(P)$ is a random variable.
\cite{candes2020phase} circumvent this by showing that $\delta(P)/n$ converges in probability to some deterministic expression $h^*$ (see \eqref{eq_candes_h}), allowing for an asymptotic argument.
In Theorem \ref{thm_bound_dimension}, we deal with the issue as follows: we show that the random cone $P$ is contained in a larger random cone, which only takes one of $n$ possible shapes.
For the shapes likely to occur, we show that the condition in \eqref{eq_approx_kin_up} is satisfied.
We show that the others occur with small probability.

In conclusion, we can calculate the probability of linear separability with the intrinsic volumes of polyhedral cones such as $\spa\{v\}\oplus [0,\infty)^n$.
Furthermore, we can give good approximations of this probability, by calculating the statistical dimension $\delta(P)$.
The problem remains to calculate the intrinsic volumes of $P$.

\subsection{Contributions}\label{sec_int_cont}

We provide new results about the probability of linear separability, the intrinsic volumes of the polyhedral cone $\spa\{v\}\oplus[0,\infty)^n$, and Youden's demon problem.
For the probability of linear separability, we provide an identity for the case that the features $x$ are Gaussian, and the labels $y$ depend on $x$ only through some marginal $x^T\beta^*$, as is the case in generalized linear models (see \eqref{eq_dependence_xTbeta}).
We also provide two upper bounds, which complement and improve existing results.
To achieve these results, we calculate the intrinsic volumes of polyhedral cones of the form $\spa\{v\}\oplus [0,\infty)^n$.
In this calculation, we run into a version of Youden's demon problem, for which we provide a new identity.

\subsubsection*{A formula for the probability of linear separability.}
With Theorem \ref{thm_prob_separation}, we provide a formula for the probability of linear separability, given Gaussian features $x$ and labels $y$ depending on $x$ only through some marginal $x^T\beta^*$.
This is the case, for example, in generalized linear models (see \eqref{eq_dependence_xTbeta}).
It was already known that the probability of linear separability without intercept can be expressed as a sum of the intrinsic volumes of the polyhedral cone $\spa\{v\}\oplus[0,\infty)^n$, see Section \ref{subsub_kinematic}.
Our contribution lies in providing an expression for these intrinsic volumes, in Theorem \ref{thm_intrinsic}.
The formula simplifies for the case of sign-flip noise (see \eqref{eq_dependence_xTbeta}, Remark \ref{rem_sign_flip}) or the signal-less imbalanced model (see \eqref{eq_intercept}, Proposition \ref{prop_prob_intercept}), and is particularly simple for the sign-flip noise model in dimension $p=2$ (see Corollary \ref{cor_dim2}).

At this point, we remark that we calculated the probability of linear separability without intercept (Theorem \ref{thm_prob_separation}), and with intercept but without signal (Proposition \ref{prop_prob_intercept}).
We do not provide a formula for the probability of linear separability with both intercept and dependence of $y$ and $x$.
In this case, the intrinsic volumes of a polyhedral cone of the form $\spa\{y\}\oplus\spa\{v\}\oplus[0,\infty)^n$ must be calculated.
We leave this to further research.

\subsubsection*{The intrinsic volumes of the polyhedral cone $\spa\{v\}\oplus[0,\infty)^n$.}
Our formulas for the probability of linear separability in Section \ref{sec_fb_formula} rely on the intrinsic volumes of polyhedral cones of the form $P:=\spa\{v\}\oplus[0,\infty)^n$.
We calculate these in Theorem \ref{thm_intrinsic}.
To do so, we establish an inequality description of the polyhedral cone in Theorem \ref{thm_description} and describe the projection onto it in Algorithm \ref{alg_2}.

\subsubsection*{A formula for Youden's demon problem.}
Youden's demon problem for Gaussian $g\sim\mathcal{N}(0,I_m)$ is the problem of calculating the probability that the sample mean is larger than $k$ observations and smaller than the remaining $m-k$.
In the calculation of the intrinsic volumes of polyhedral cones of the form $\spa\{v\}\oplus[0,\infty)^n$, we run into a version of this problem.
Although recursive relations for this problem were already available (see Appendix \ref{sec_youden_hist}), we provide a closed-form expression, in terms of the complex Gaussian cumulative distribution function (Theorem \ref{cor_youden_general}).
To prove this result, we leverage and extend recent results from \cite{kabluchko2020absorption}.

\subsubsection*{Upper bounds for the probability of linear separability.}
In this work, we provide two new upper bounds for the probability of linear separability, Theorem \ref{thm_bound_dimension} and Proposition \ref{prop_bound_sign_flip}.
These results are compared in Table \ref{tab_comparison} below.
For a more complete review see Section \ref{sec_int_existing}. 

\begin{table}[htb]
\centering
\resizebox{\textwidth}{!}{
\begin{tabular}{|l|c|c|c|c|c|} 
\hline 
 & \eqref{eq_satoshi_upper} & \eqref{eq_chardon_1}
 & \eqref{eq_chardon_2}
 & Theorem \ref{thm_bound_dimension} & Proposition \ref{prop_bound_sign_flip} \\
\hline 
Any GLM & Yes & No & Yes & Yes & No \\
Only Gaussian & No & Yes & No & Yes & Yes \\
HD asymptotics & $\sigma\gtrsim \frac{p}{n}$ 
& $n\gtrsim Bp$ 
& $n\gtrsim B\log^4(B)p$ 
& $\sigma\gtrsim \frac{p}{n}$ 
& $\sigma\gtrsim \frac{p}{n}\log\frac{n}{p}$ \\
Classical asymp. 
& No 
& Yes
& Yes
& Yes
& Yes \\
Explicit constants & Yes & No & No & Yes & Yes \\
Sharp at $\sigma=1/2$ & No & No & No & No & Yes \\
\hline 
\end{tabular}
}
\caption{Finite sample upper bounds for the probability of linear separability.
\eqref{eq_satoshi_upper} refers to \cite*[Theorem 14]{hayakawa2021estimating}, \eqref{eq_chardon_1} refers to \cite[Theorem 1]{chardon2024finite},
\eqref{eq_chardon_2} refers to \cite[Theorem 4]{chardon2024finite} assuming $B\in o(p)$ to reduce notation,
Theorem \ref{thm_bound_dimension} and Proposition \ref{prop_bound_sign_flip} are new contributions.
\\
Here, $\sigma:=\mathbb{P}[yx^T\beta^*<0]$ denotes the probability of observing a `wrong label', see \eqref{eq_oracle} and \eqref{eq_defn_delta}. We use the notation $B:=\max\{\|\beta^*\|_2,e\}$.
`\textit{Any GLM}': The result applies to all distributions of the labels $y$ depending only on a marginal of $x$, see \eqref{eq_dependence_xTbeta}.
`\textit{Only Gaussian}': The results rely on the assumption that $x$ are Gaussian.
`\textit{HD asymptotics}': The upper bounds vanish for large $p,n\rightarrow\infty$, and $p/n\rightarrow \kappa$ tending to a constant $\kappa\in(0,0.5)$, under the mentioned condition. 
Here, $\gtrsim$ means larger up to a universal constant.
`\textit{Classical asymptotics}': The bound shows that as $n$ tends to infinity, keeping all else fixed, the probability of linear separability vanishes. 
`\textit{Explicit constants}': The constants appearing in the results are reasonably small/large.
`\textit{Sharp at $\sigma=1/2$}': The upper bound is an equality at $\sigma=1/2$, i.e. it recovers \eqref{eq_cover} (Cover's identity).
}
\label{tab_comparison}
\end{table}
Our new contribution Theorem \ref{thm_bound_dimension} stands out as follows:
As of now, it is the most widely applicable result which achieves a vanishing upper bound given $\sigma\gtrsim p/n$ in the high-dimensional asymptotic regime, which also shows that the probability of linear separability vanishes as $n\rightarrow\infty$.
The new Proposition \ref{prop_bound_sign_flip} is the only available upper bound that recovers the equality at $\sigma=1/2$, making it the sharpest approximation for $\sigma\approx 1/2$.
Both results have explicit and reasonably sized constants.

\section{A formula and two bounds}

In this section, we provide the main results of this work.
First, we provide a formula for the probability of linear separability in Section \ref{sec_fb_formula}.
Then, we provide two bounds for the probability of linear separability in Section \ref{sec_fb_bounds}.

\subsection{A formula for the probability of linear separability}\label{sec_fb_formula}
Here, we provide a formula for the probability of linear separability.
It is based on the kinematic formula, see Section \ref{subsub_kinematic}.
Consequentially, in this section, we assume that the features $x_j$ follow a Gaussian distribution and that the labels $y$ only depend on $x$ through some marginal $x^T\beta^*$ (see Section \ref{sec_models} for a discussion, in particular \eqref{eq_dependence_xTbeta}).
We first introduce the general formula in Section \ref{sec_fb_f_general}.
Then, we provide two examples.
The first one is the two-dimensional case in Section \ref{sec_fb_f_2d}, the second one is the signal-less case with class (im-)balance (see \eqref{eq_intercept}) in Section \eqref{sec_fb_f_imbalanced}.
In the latter, we also show that the general formula recovers Cover's identity \eqref{eq_cover} if the labels $y$ and features $x$ are independent and there is class balance.

\subsubsection{General case}\label{sec_fb_f_general}
Without further ado, we state the formula for the probability of linear separability.
The notation is explained below.

\begin{thm}\label{thm_prob_separation}
Suppose that $x_j\sim\mathcal{N}(0,I_p)$ and the distribution of $y_j$ only depends on $x_j$ through $x_j^T\beta^*$ for some $\beta^*\in \mathbb{R}^p\setminus\{0\}$.
Let $\delta:=\mathbb{P}[yx^T\beta^*>0]$, define $v_j:=y_jx_j^T\beta^*/\|\beta^*\|_2$ and $P:=\spa\{v\}\oplus[0,\infty)^n$.
Let $A_N$ be the event that the first $N$ elements of $v$ are positive, and the rest negative.
The probability of linear separability without intercept is:
\[
S(p,n)=
\delta^n+(1-\delta)^n+2\sum_{N=1}^{n-1}\sum_{j\geq 1\text{ odd}}{n\choose N}\delta^N(1-\delta)^{n-N}\nu_{n-p+1+j}(P|A_N),
\]
where $\nu_{k}(P|A_N)$ denotes the expected $k$-th intrinsic volume of $P$ given $A_N$.
It holds that $\nu_0(P|A_N)=0$. For $k\in\{1,\ldots,n-1\}$,
\[
\nu_{n-k}(P|A_N)
=\sum_{\substack{(L,R)\subset\mathcal{I}_N\\|L|+|R|=k+1}}
\mathbb{E}\left[\prod_{j\in L\cup R}\Phi\left(\frac{igv_j}{\|v\|_{LR}}\right)\prod_{j\in[n]\setminus(L\cup R)}\Phi\left(\frac{g'v_j}{\|v\|_{LR}}\right)\Bigg|A_N\right],
\]
\[
\nu_n(P|A_N)=\mathbb{P}\left[\min_{1\leq l\leq N}\frac{g_j}{v_j}\geq \max_{N+1\leq r\leq n}\frac{g_r}{v_r}\Bigg|A_N\right]=1-\sum_{k=1}^{n-1}\nu_k(P|A_N).
\]
\end{thm}

We recall some of the notation used in Theorem \ref{thm_prob_separation}.
We use $i:=\sqrt{-1}$ to denote the imaginary number.
The formula uses the complex Gaussian cumulative distribution function, which is defined as:
\[
\Phi:\mathbb{C}\rightarrow\mathbb{C},\quad
z\mapsto \frac{1}{2}+\frac{1}{\sqrt{2\pi}}\int_0^ze^{-\frac{t^2}{2}}dt.
\]
The formula could be expressed without using complex numbers (see Remark \ref{rem_real}), although this would make it less compact.
For $N\in\{1,\ldots,n-1\}$, we define the index set $\mathcal{I}_N:=\{1,\ldots,N\}\times\{N+1,\ldots,n\}$.
For $v\in\mathbb{R}^n$, and a pair of sets $(L,R)\subset\mathcal{I}_N$ we define:
\[
\|v\|_{LR}^2:=\sum_{j\in L\cup R}v_j^2.
\]
We recall that the intrinsic volumes $\nu_k(P|A_N)$ can be expressed as a collection of inequalities between random variables, which may be interesting besides the expression in Theorem \ref{thm_prob_separation}, see Remark \ref{rem_combinatorics}.

\begin{rem}\label{rem_sign_flip}
For sign-flip noise, the random variables $v_j$ conditional on $A_N$ have the same distribution as the unconditional $|v_j|$ if $j\leq N$, and as $-|v_j|$ if $j>N$, i.e. as the absolute value of a Gaussian.
This may simplify the expression considerably.
For example, it follows that:
\begin{equation}\label{eq_intrinsic_n}
\nu_n(P|A_N)=\mathbb{P}\left[\min_{1\leq l\leq N}\frac{g_j}{|v_j|}\geq \max_{N+1\leq r\leq n}\frac{g_r}{|v_r|}\right]=\frac{1}{{n\choose N}}.
\end{equation}
We exploit this in the two-dimensional example below.
\end{rem}

\begin{proof}[Proof of Theorem \ref{thm_prob_separation}]
By Proposition \ref{prop_candes},
\[
S(n,p)=\mathbb{P}\left[\theta L_{p-1}\cap \spa\{v\}\oplus[0,\infty)^n\neq \{0\}\right].
\]
We condition on $A_N$ and $v$.
Note that if $N\in\{0,n\}$, then $P=\mathbb{R}^n$.
Then,
\[
\mathbb{P}\left[P\cap L_{p-1}\neq \{0\}\right]
=\sum_{N=0}^n{n\choose N}\delta^N(1-\delta)^{n-N}\mathbb{E}\left[\mathbb{P}\left[P\cap L_{p-1}\neq \{0\}\Big|v\right]\Big|A_N\right].
\]
Note that if $N\in\{0,n\}$, then $P=\mathbb{R}^n$.
Else, by Theorem \ref{thm_kinematic}:
\[
\mathbb{E}\left[\mathbb{P}\left[P\cap L_{p-1}\neq \{0\}\Big|v\right]\Big|A_N\right]
=2\sum_{j\geq 1\text{ odd}}\mathbb{E}\left[\nu_{n-p+1+k}(P|v)\Big|A_N\right]
\]
Here, we denote by $\nu_k(P|v)$ the $k$-th intrinsic volume of $P$ conditional on $v$, given in Theorem \ref{thm_intrinsic}.
For $\nu_n(P|A_N)$, we note that for $g\sim\mathcal{N}(0,I_n)$,
\[
\nu_n(P|A_N)
=\mathbb{E}\left[\mathbb{P}\left[g\in P\Big|v\right]\Big|A_N\right]
=\mathbb{P}\left[\min_{1\leq l\leq N}\frac{g_j}{v_j}\geq \max_{N+1\leq r\leq n}\frac{g_r}{v_r}\Bigg|A_N\right].
\]
The proof is complete.
\end{proof}

\subsubsection{Example for dimension two with sign-flip noise}\label{sec_fb_f_2d}
In the two-dimensional case $p=2$, with sign-flip noise, the formula in Theorem \ref{thm_prob_separation} gives a simple expression.

\begin{cor}\label{cor_dim2}
Suppose that $x_j\sim\mathcal{N}(0,I_2)$ and $y$ follows the sign-flip noise model \eqref{eq_sign_flip}. Then, the probability of linear separability without intercept is:
\[
S(n,2)=\delta^n+(1-\delta)^n+2\sum_{N=1}^{n-1}\delta^N(1-\delta)^{n-N}
\]
\end{cor}

The result follows from Theorem \ref{thm_prob_separation}, using \eqref{eq_intrinsic_n}. 
For the case where $\delta=1/2$, this recovers Cover's identity \eqref{eq_cover}, as $S(n,2)=n/2^{n-1}$.

In low dimensions, the probability of linear separability is considerably easier to calculate than in high dimensions (in particular see \cite{bruckstein1985monotonicity} for a simple argument in dimension 1).
Similarly, Corollary \ref{cor_dim2} can be proved directly, without using the kinematic formula.
Moreover, one can directly prove the same identity for $S_0(n,1)$.

\subsubsection{The signal-less (im)balanced case}\label{sec_fb_f_imbalanced}

The sign-flip noise model leads to simplifications of the formula in Theorem \ref{thm_prob_separation} since the $|v_j|$ are independent of $A_N$ (see Remark \ref{rem_sign_flip}).
The signal-less model with class imbalance \eqref{eq_intercept} simplifies even more, as $|v_j|=|y_j|\equiv 1$, see Proposition \ref{prop_candes}.
The notation $\Phi_k$ is explained below.

\begin{prop}\label{prop_prob_intercept}
Suppose that $x_j\sim\mathcal{N}(0,I)$, that $y_j$ is independent of $x_j$.
Let $b:=\mathbb{P}[y_j=1]$.
The probability of linear separability with intercept is:
\[
S_0(p,n)=
b^n+(1-b)^n+2\sum_{j\geq 1\text{ odd}}\mu_{n-p+j},
\]
where $\mu_n=\sum_{N=1}^{n-1}b^N(1-b)^{n-N}$, $\nu_0=0$ and for $k\in\{1,\ldots,n-1\}$, and for $(g,g')\sim\mathcal{N}(0,I_2)$:
\[
\mu_{n-k}:=
\]
\[{n\choose k+1}\mathbb{E}\left[\left(b\Phi_k(ig)+(1-b)\Phi_k(-ig)\right)^{k+1}\left(b\Phi_k(g')+(1-b)\Phi_k(-g')\right)^{n-(k+1)}\right].
\]
\end{prop}

We used the notation:
\[
\Phi_k(g):=\frac{1}{2}+\frac{1}{\sqrt{2\pi}}\int_{0}^{g/\sqrt{k+1}}e^{-\frac{t^2}{2}}dt.
\]
The identity follows as Theorem \ref{thm_prob_separation} and binomial theorems. 
The proof is provided in Appendix \ref{app_prop_prob_intercept}.
In the case that $b=1/2$, using Pascal's triangle one can verify that Proposition \ref{prop_prob_intercept} recovers Cover's identity with intercept.

\subsection{Two bounds for the probability of linear separability}\label{sec_fb_bounds}

Here, we provide two bounds for the probability of linear separability.
The first one in Section \ref{sec_fb_b_dim} uses the kinematic formula (see Section \ref{subsub_kinematic}), and our inequality description of the cone $\spa\{v\}\oplus[0,\infty)^n$ (Theorem \ref{thm_description}).
The second one in Section \ref{sec_fb_b_sf} uses that for sign-flip noise, the cone $\spa\{v\}\oplus[0,\infty)^n$ only depends on the signal strength through $N$, the number of positive indices of $v$.
We compare Theorem \ref{thm_bound_dimension} and Proposition \ref{prop_bound_sign_flip} to existing bounds in Section \ref{sec_int_cont}.

\subsubsection{Upper bound using the statistical dimension}\label{sec_fb_b_dim}
Below, we state an upper bound for the probability of linear separability.
Its proof relies on the approximation of the kinematic formula \eqref{eq_approx_kin_up}, and an approximation of the polyhedral cone $\spa\{v\}\oplus [0,\infty)^n$, see Lemma \ref{lem_bound_P}.

\begin{thm}\label{thm_bound_dimension}
Suppose that $x_j\sim\mathcal{N}(0,I_p)$, and that for some $\beta^*\in S^{p-1}$, the labels $y_j$ only depend on $x_j$ through $x_j^T\beta^*$.
Suppose that for some $t\geq0$, 
\begin{equation}\label{eq_condition_delta}
\frac{2(p-1)}{n}+\left(4+\frac{1}{\sqrt{2}}\right)\sqrt{\frac{t}{n}}\leq \mathbb{P}[yx^T\beta^*<0]\leq\frac{1}{2}.
\end{equation}
Then, the probability of linear separability without intercept is at most:
\[
S(n,p)\leq 3\exp\left(-t\right).
\]
\end{thm}

The proof is provided in Appendix \ref{app_bound_dimension}.
The bound resembles the bound in \eqref{eq_satoshi_upper}, we give a comparison in Section \ref{sec_int_cont}.
The same proof strategy can be applied to derive a bound for the signal-less imbalanced model, as we now show in Remark \ref{rem_bound_imbalanced}.

\begin{rem}\label{rem_bound_imbalanced}
Note that as in Section \ref{sec_fb_f_imbalanced}, we could apply the argument in Theorem \ref{thm_bound_dimension} to the signal-less imbalanced model \eqref{eq_intercept}.
By doing so, we arrive at a similar inequality, namely that if for some $t\geq 0$,
\[
\frac{2p}{n}+\left(4+\frac{1}{\sqrt{2}}\right)\sqrt{\frac{t}{n}}\leq \min\{1-b,b\}
\]
then, the probability of linear separability with intercept is at most:
\[
S_0(n,p)\leq 3\exp\left(-t\right).
\]
\end{rem}

\subsubsection{Tight bound for sign-flip noise}\label{sec_fb_b_sf}

Here, we provide an upper bound for the probability of linear separability given Gaussian features and sign-flip noise.
Although Proposition \ref{prop_bound_sign_flip} does not rely on the kinematic formula, it exploits a common theme in our discussion:
The structure of the problem is determined by the number of `correct labels' $N$, i.e. the number of $y_jx_j^T\beta^*$ which are positive.
The proof strategy we employ also leads to lower bounds, see \eqref{eq_lowerb_sign_flip}.

\begin{prop}\label{prop_bound_sign_flip}
Suppose that $x_j\sim\mathcal{N}(0,I)$.
Under the sign flip noise model \eqref{eq_sign_flip} with $\delta\geq 1/2$, the probability of linear separability without/with intercept satisfies:
\[
S(n,p)
\leq 2\cdot\delta^n\sum_{j=0}^{p-1}{n-1\choose j},
\quad
S_0(n,p)
\leq 2\cdot\delta^n\sum_{j=0}^{p}{n-1\choose j}.
\]
Under the signal-less imbalanced model \eqref{eq_intercept} with $b\geq 1/2$, the probability of linear separability with intercept satisfies:
\[
S_0(n,p)
\leq 2\cdot b^n\sum_{j=0}^{p}{n-1\choose j}.
\]
\end{prop}

Proposition \ref{prop_bound_sign_flip} is proved in Appendix \ref{app_bounds}.
We note that all three inequalities are sharp, as they recover Cover's identity \eqref{eq_cover} in the case $\delta=1/2$, respectively $b=1/2$.
We show that the rate $p/n\log (en/p)$ holds for the probability of linear separability without intercept (the same argument can be applied to the other two cases).
Note that:
\[
2\delta^{n}\sum_{i=0}^{p-1}{n-1\choose i}\leq 2\exp\left(-n(1-\delta)\right)\left(\frac{e(n-1)}{p-1}\right)^{p-1}.
\]
Consequentially,
\begin{equation}\label{eq_bound}
1-\delta\geq \frac{p-1}{n}\log\frac{e(n-1)}{p-1}+\frac{t}{n}\quad \Rightarrow\quad S(n,p)\leq 2\exp(-t).
\end{equation}

We compare Proposition \ref{prop_bound_sign_flip} and the rate \eqref{eq_bound} with the other available upper bounds in Section \ref{sec_int_cont}.

\begin{rem}\label{rem_sign_flip_lowerb}
Arguing as in the proof of Proposition \ref{prop_bound_sign_flip}, we can provide lower bounds of the following form:
\begin{equation}\label{eq_lowerb_sign_flip}
2\cdot(1-\delta)^n\sum_{j=0}^{p-1}{n-1\choose j}\leq 
S(n,p).
\end{equation}
These lower bounds are sharp, as they recover Cover's identity \eqref{eq_cover} at $\delta=1/2$.
Yet, they are worse than the universal lower bound by \cite[Corollary 3.7]{wagner2001continuous}.
\end{rem}

\section{Properties of the cone $\spa\{v\}\oplus [0,\infty)^n$}

The probability of linear separability with Gaussian features is a sum of the intrinsic volumes of a polyhedral cone, see Section \ref{subsub_kinematic}.
This cone is defined as follows: For a $v\in(\mathbb{R}\setminus \{0\})^n$,
\begin{equation}\label{eq_defn_P}
P:=\{\lambda v+u:\lambda\in\mathbb{R},u\in[0,\infty)^n\}=\spa\{v\}\oplus[0,\infty)^n.
\end{equation}
In this section, we study this polyhedral cone in depth.
We provide an inequality description of $P$ in Section \ref{sec_description}.
Next, we describe the projection of a vector in $\mathbb{R}^n$ onto $P$ in Section \ref{sec_projection}.
Finally, we calculate the intrinsic volumes of $P$, i.e. the probability that a Gaussian projected onto $P$ lands on a face of a certain dimension, in Section \ref{sec_intrinsic}.
This last section will require a formula for a version of Youden's demon problem, which we provide in Section \ref{sec_youden}.

Throughout this section, $v\in(\mathbb{R}\setminus \{0\})^n$ is treated as a fixed, deterministic vector.
This contrasts with other parts of this work, where $v_j$ is defined as $y_jx_j^T\beta^*/\|\beta^*\|_2$, making it a random variable.
Moreover, we will later use $x,y$ to denote arbitrary elements of $\mathbb{R}^n$.
It may help to keep in mind that Sections \ref{sec_description} and Section \ref{sec_projection} are entirely deterministic.

\subsection{Inequality description}\label{sec_description}

In this subsection, we provide an inequality description for the interpolation cone.
Throughout, we define the cone $P\subset\mathbb{R}^n$ as in \eqref{eq_defn_P}, for some fixed $v\in\mathbb{R}^n$ with no zero entries: $v_i\neq 0$ for all  $i\in[n]$.
We start with Lemma \ref{lem_polar}, which gives us a first understanding of the faces of $P$.
The polar of a closed convex cone $C\subset\mathbb{R}^n$ is defined as:
\[
C^\circ:=\bigcap_{c\in C}\{x\in \mathbb{R}^n: x^Tc\leq 0\}.
\]
In Lemma \ref{lem_polar}, we describe the polar of our polyhedral cone $\spa\{v\}\oplus [0,\infty)^n$.

\begin{lem}\label{lem_polar}
\[
\left(
[0,\infty)^n\oplus \spa\{v\}
\right)^\circ
=
(-\infty,0]^n\cap \spa\{v\}^\perp.
\]
\end{lem}

Lemma \ref{lem_polar} is a consequence of a known identity about the polar of convex cones.
Namely, if $C_1,C_2$ are closed convex cones, then $(C_1\oplus C_2)^\circ = C_1^\circ\cap C_2^\circ$, see e.g. \cite[Theorem 1.6.9]{schneider2014convex}.
Yet, we provide a short proof below.

\begin{proof}
Choose any $u\in[0,\infty)^n$, any $\alpha\in\mathbb{R}$ and any $x\in (-\infty,0]^n\cap \spa\{v\}^\perp$.
It holds that:
\[
\langle u+\alpha v,x\rangle
=\langle u,x\rangle \leq 0
\]
This shows the inclusion $\supseteq$.
We establish the inclusion $\subseteq$ by contraposition.
Suppose that $x\not\in(-\infty,0]^n$.
So, there is an index $j\in[n]$, such that $x_j>0$.
Let $e_j\in[0,\infty)^n$ denote the $j$-th unit vector.
We have $\langle x,e_j\rangle =x_j>0$, so $x$ is not in the polar set of $[0,\infty)^n\oplus\spa\{v\}$.
Now take any $x\not\in \spa\{v\}^\perp$.
Since $v$ and $-v$ both lie in $\spa\{v\}$ and either $\langle x,v\rangle >0$ or $\langle x,-v\rangle >0$, this $x$ is not in the polar set either.
The proof is complete.
\end{proof}

We now introduce the inequality description of the interpolation cone $P$.
The key idea is the splitting of the indices $[n]$ into two sets, say the `left' and the `right' set.
The elements of $v$ indexed by the left side are positive, and the elements indexed by the right side are negative.
Without loss of generality, we assume that the values $v_l>0$ are positive if and only if their indices lie in $l\in\{1,\ldots,N\}$, whereas the values $v_r$ indexed by $r\in\{N+1,\ldots,n\}$ are negative, i.e. $v_r<0$.
We will see that any inequality describing $P$ is of the form:
\[
-\frac{x_l}{v_l}+\frac{x_{r}}{v_{r}}=-\frac{x_l}{|v_l|}-\frac{x_{r}}{|v_{r}|}\leq 0,
\]
for some $(l,r)\in\mathcal{I}_N:=\{1,\ldots,N\}\times\{N+1,\ldots,n\}$.
Note that the vector $a^{lr}:=-e_l/|v_l|-e_r/|v_r|$  lies in $\ker\{v\}\cap(-\infty,0]^n$.
In view of Lemma \ref{lem_polar}, this suggests that we are on the right track.
In Theorem \ref{thm_description} we see that all these inequality descriptions are necessary, and that indeed each of these inequalities defines a facet of $P$.

\begin{thm}\label{thm_description}
Fix $v\in \mathbb{R}^n$ and $N\in\{1,\ldots,n-1\}$.
If $v_l>0$ for $l\in\{1,\ldots,N\}$ and $v_r<0$ for $r\in\{N+1,\ldots,n\}$, then, an inequality description for the polyhedral cone $P:=\spa\{v\}\oplus[0,\infty)^n$ is:
\begin{equation}\label{eq_description}
P=\bigcap_{(l,r)\in\mathcal{I}_N}\left\{
x\in\mathbb{R}^n:\frac{x_l}{v_l}\geq \frac{x_r}{v_r}
\right\}.
\end{equation}
Moreover, every inequality in this description of $P$ is facet-defining.
\end{thm}

Another way of writing this inequality description of $P$ is:
\[
P=\left\{x\in\mathbb{R}^n:\min_{1\leq l\leq N}\frac{x_l}{v_l}\geq \max_{N+1\leq r\leq n}\frac{x_r}{v_r}\right\}.
\]
Visualizing the cone $P$ is not straightforward, as the nature of its structure only shows in higher dimensions.
It may be helpful to view it as a bipartite graph, see Figure \ref{fig_bipartite}.
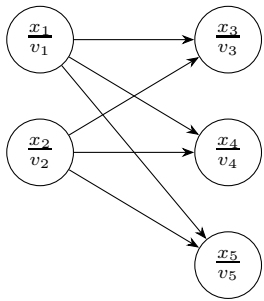
\begin{figure}[htb]
    \centering
    \begin{tikzpicture}[>=Stealth, auto, node distance=1.5cm]
    
    \node[draw, circle] (x1l) {$\frac{x_1}{v_1}$};
    \node[draw, circle, below of=x1l] (x1l2) {$\frac{x_2}{v_2}$};
    
    \node[draw, circle, right of=x1l, node distance=2.5cm] (x1r1) {$\frac{x_3}{v_3}$};
    \node[draw, circle, below of=x1r1] (x1r2) {$\frac{x_4}{v_4}$};
    \node[draw, circle, below of=x1r2] (x1r3) {$\frac{x_5}{v_5}$};
    
    \draw[->] (x1l) -- (x1r1);
    \draw[->] (x1l) -- (x1r2);
    \draw[->] (x1l) -- (x1r3);
    \draw[->] (x1l2) -- (x1r1);
    \draw[->] (x1l2) -- (x1r2);
    \draw[->] (x1l2) -- (x1r3);
    \end{tikzpicture}
    \caption{
    The inequalities in the description of the cone $P$ in Theorem \ref{thm_description} form a complete bipartite graph.
    In this example, $N=2$ and $n-N=3$. Each node represents an index in $\{1,\ldots,n\}$, and each edge represents an inequality in the description of $P$. The arrow represents the direction of the inequality, as the constraints of $P$ can be written as $x_1/v_1\geq x_3/v_3$ et cetera.}
    \label{fig_bipartite}
\end{figure}

Despite the description of $P$ given in Theorem \ref{thm_description}, we will see that working with $P$ remains challenging.
It may help to use the following approximation, for which the intrinsic volumes are easy to calculate.
We will use this result to prove the upper bound in Theorem \ref{thm_bound_dimension}.

\begin{lem}\label{lem_bound_P}
Fix $v\in \mathbb{R}^n$ and $N\in\{1,\ldots,n-1\}$.
If $v_l>0$ for $l\in\{1,\ldots,N\}$ and $v_r<0$ for $r\in\{N+1,\ldots,n\}$, then:
\[
\spa\{v\}\oplus[0,\infty)^n\subset \left(\mathbb{R}^{N}\times[0,\infty)^{n-N}\cup[0,\infty)^{N}\times\mathbb{R}^{n-N}\right).
\]
\end{lem}

\begin{proof}
We proceed by contraposition.
Take any $x$ not contained in the set $\mathbb{R}^{N}\times[0,\infty)^{n-N}\cup[0,\infty)^{N}\times\mathbb{R}^{n-N}$.
There exists an $l\in\{1,\ldots,N\}$ and an $r\in\{N+1,\ldots,n\}$, such that $x_l<0$ and $x_r<0$.
Therefore, 
\[
\min_{1\leq l\leq N}\frac{x_l}{|v_l|}+
\min_{N+1\leq r\leq n}\frac{x_r}{|v_r|}<0.
\]
By Theorem \ref{thm_description}, $x\not\in \spa\{v\}\oplus[0,\infty)^n$.
The proof is complete.  
\end{proof}

The proof of Theorem \ref{thm_description} is given in Appendix \ref{app_description}.
The inclusion $\supset$ takes work, while the inclusion $\subset$ is straightforward.

\subsection{Projection}\label{sec_projection}

Theorem \ref{thm_description} provides an inequality description of the interpolation cone $P$.
Next, we find the projection onto $P$.
This takes a bit of work, so we first give an example.

\begin{exa}\label{exa_projection}
Suppose $n=5$ and $v_1,v_2>0>v_3,v_4,v_5$.
That is, $N=2$.
If $x\in \mathbb{R}^n$ is such that $x_i/v_i>x_j/v_j$ for all $(i,j)\in\mathcal{I}_N=\{1,2\}\times\{3,4,5\}$, then $x\in P$.
So, $x$ is equal to its projection onto $P$. 
This is the situation illustrated in Figure \ref{fig_bipartite}.
Now consider a situation where some constraints are violated, such as:
\[
\frac{x_5}{v_5}<
\mathbf{\frac{x_2}{v_2}<
\frac{x_4}{v_4}<
\frac{x_3}{v_3}<}
\frac{x_1}{v_1}
\]
Each of the indices $2,3,4$ is contained in an inequality, for which a constraint is violated. 
This is illustrated in Figure \ref{fig_bipartite_violated}.

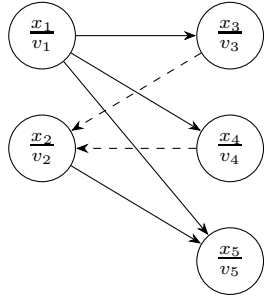
\begin{figure}[htb]
    \centering
    \begin{tikzpicture}[>=Stealth, auto, node distance=1.5cm]
    
    \node[draw, circle] (x1l) {$\frac{x_1}{v_1}$};
    \node[draw, circle, below of=x1l] (x1l2) {$\frac{x_2}{v_2}$};
    
    \node[draw, circle, right of=x1l, node distance=2.5cm] (x1r1) {$\frac{x_3}{v_3}$};
    \node[draw, circle, below of=x1r1] (x1r2) {$\frac{x_4}{v_4}$};
    \node[draw, circle, below of=x1r2] (x1r3) {$\frac{x_5}{v_5}$};
    
    \draw[->] (x1l) -- (x1r1);
    \draw[->] (x1l) -- (x1r2);
    \draw[->] (x1l) -- (x1r3);
    \draw[->,dashed] (x1r1) -- (x1l2);
    \draw[->,dashed] (x1r2) -- (x1l2);
    \draw[->] (x1l2) -- (x1r3);
    \end{tikzpicture}
    \caption{A visual representation of the inequalities in Example \ref{exa_projection}.
    The arrows pointing right symbolize inequalities in the description of $P$ that are satisfied.
    The dashed arrows pointing left show the violated constraints.}
    \label{fig_bipartite_violated}
\end{figure}

Informally, the projection maps $x$ into $P$, with the least amount of `effort' in terms of the Euclidean distance.
Here is a naive attempt to project onto $P$: 
`Average' the values of $x$ at nodes with violated constraints, and leave the others unchanged.
This gives us a candidate for the projection, which we call $y:=(x_1,v_1\Sigma,v_2\Sigma,v_3\Sigma,x_5)$, where $\Sigma$ is an average with appropriately chosen weights.
This $y$ indeed lies in $P$ by Lemma \ref{lem_projection_in_P}.
While $y$ may coincide with the projection (see Example \ref{exa_projection_success}) this is not always the case (see Example \ref{exa_projection_fail}).

A slightly more sophisticated procedure is needed (see Algorithm \ref{alg_2}).
Instead of averaging over all nodes that violate the constraints, we only average over some, leaving the other indices unchanged.
Such a subset must include the `worst offenders' from both sides of the graph, otherwise the inequalities cannot be satisfied.
To see what these `worst offenders' would be, note that we want the values on the left side to be larger than the values on the right side. 
So, the `worst offenders' on the left are the smallest values on the left, while the ones on the right are the largest values.
In our example, this would be $x_2/v_2$ on the left and $x_3/v_3$ on the right.
Algorithm \ref{alg_2} first checks if after averaging those, the candidate vector lies in $P$.
If not, the next smallest values from the left and the next largest values from the right are added, until the vector lies in $P$.
It turns out that it does not matter whether the offenders are first added from the left or the right, as adding them from the left (right) only decreases (increases) the average, leaving the remaining set of offenders unchanged (Lemma \ref{lem_average_monotone}).
In Theorem \ref{thm_alg_2_correct}, we show that Algorithm \ref{alg_2} indeed produces the projection of $x$ onto $P$.
Consequentially, the projection is obtained by averaging the nodes corresponding to the indices $2,3$, or $2,3,4$, depending on the values of the ratios $x_i/v_i$.
\end{exa}

We proceed as follows: 
We first construct a collection of vectors $y^{LR}(x)$.
We call these the `candidate projections'.
Later we will see that one candidate is indeed the projection $x$ onto $P$.
First, we show that the most restrictive one lies in $P$ (Lemma \ref{lem_projection_in_P}).
Next, we introduce Algorithm \ref{alg_2}, which constructs the projection of $x$ onto $P$.
This is proved in Theorem \ref{thm_alg_2_correct}, where we leverage Lemma \ref{lem_projection_in_P}.
To characterize the projection of $x$ onto $P$, we make use of the following proposition.
A proof can be found in \cite[Proposition 3.2.3]{hiriart2013convex}.

\begin{prop}\label{prop_projection}
Let $K$ be a closed convex cone. Then, $y\in\mathbb{R}^n$ is the projection of $x\in\mathbb{R}^n$ onto $K$, if and only if all of the following conditions are met:
\begin{enumerate}
\item $y\in K$,
\item $x-y\in K^\circ$,
\item $\langle x-y,y\rangle =0$.
\end{enumerate} 
\end{prop}

\subsubsection*{Construction of the candidate projections}
Here, we construct the candidates for the projection of a vector $x$ onto $P$, which we denote by $y$.
We show in Lemma \ref{lem_projection_in_P} that one of them lies in $P$.
It may or may not be the projection of $x$ onto $P$, see Examples \ref{exa_projection_success} and \ref{exa_projection_fail}.
We first introduce some necessary notation.
Recall that $\mathcal{I}_N:=\{1,\ldots,N\}\times\{N+1,\ldots,n\}$.
Fix $v\in\mathbb{R}^n$ and $N\in\{1,\ldots,n-1\}$ such that $v_l>0$ for $l\in\{1,\ldots,N\}$ and $v_r<0$ for $r\in\{N+1,\ldots,n\}$.
Define: 
\[
P:=\bigcap_{(l,r)\in \mathcal{I}_N}\left\{x\in\mathbb{R}^n:\frac{x_l}{v_l}\geq \frac{x_r}{v_r} \right\}.
\]
Fix any $x\in \mathbb{R}^n$.
We define $\bar{L}(x)$ as the set of indices in $\{1,\ldots,N\}$ at which $x$ violates any constraint of $P$, and similarly $\bar{R}(x)$:
\begin{equation}\label{eq_LR_bar_def}
\bar{L}(x):=\bigcup_{r=N+1}^n\left\{l\leq N:\frac{x_l}{v_l}<\frac{x_r}{v_r}\right\},\quad 
\bar{R}(x):=\bigcup_{l=1}^N\left\{r> N:\frac{x_l}{v_l}<\frac{x_r}{v_r}\right\}.
\end{equation}
We use the notation $L$ and $R$ to mimic the left and right parts of the graph visualized in Figures \ref{fig_bipartite} and \ref{fig_bipartite_violated}.
We define for any $(L,R)\subset\mathcal{I}_N$:
\begin{equation}\label{eq_defn_projection}
\Sigma_{LR}(x):=\frac{\sum_{k\in L\cup R} x_kv_k}{\sum_{k\in L\cup R}v_k^2},\quad
y_i^{LR}(x):=\begin{cases}
v_i\Sigma_{LR}(x),&\text{if }i\in L\cup R,\\
x_i,&\text{else.}
\end{cases}
\end{equation}

\begin{lem}\label{lem_projection_in_P}
Fix $x\in\mathbb{R}^n$.
If $(L,R)=(\bar{L}(x),\bar{R}(x))$, then $y^{L,R}(x)\in P$.
Moreover, if $x$ satisfies no constraint of $P$ with equality, the vector $y^{\bar{L},\bar{R}}(x)$ satisfies the inequalities with indices in $\bar{L}(x)\times \bar{R}(x)$ with equality, and all others with strict inequality.
\end{lem}

\begin{proof}
Fix any $x\in\mathbb{R}^n$.
To simplify notation, throughout this proof we write $\Sigma_{LR}:=\Sigma_{LR}(x)$ and $y:=y^{LR}(x)$, as well as $(L,R)=(\bar{L}(x),\bar{R}(x))$.
As in the proof of Theorem \ref{thm_description}, for all $(l,r)\in\mathcal{I}_N$ we use the notation $a^{lr}:=-e_l/|v_l|-e_r/|v_r|$.
We define $L^c:=\{1,\ldots,N\}\setminus L$ and $R^c:=\{N+1,\ldots, n\}\setminus R$.
To verify that $y$ satisfies all inequalities in the description of $P$, we have to check four cases, namely $(l,r)$ in $L^c\times R^c$, $L^c\times R$, $L\times R^c$ or $L\times R$.
If $(l,r)\in L^c\times R^c$ then:
\[
\langle y,a^{lr}\rangle = \langle x,a^{lr}\rangle \leq 0.
\]
Moreover, the inequality is strict if $x$ satisfies no constraint of $P$ with equality.
If $(l,r)\in L\times R$, then:
\[
\langle y,a^{lr}\rangle= -\Sigma_{LR}+\Sigma_{LR}=0.
\]
We check the case $(l,r)\in L^c\times R$.
Since $l\in L^c$, for all $j\in \{N+1,\ldots,n\}$ it holds that $x_l/v_l\geq x_j/v_j$.
On the other hand, since $R$ is not empty in this case, there exists an $r\in R$, such that for all $i\in L$, it holds that $x_i/v_i< x_r/v_r$.
In particular, for all $i\in L$, $x_i/v_i< x_r/v_r\leq x_l/v_l$.
We conclude that:
\[
\frac{y_r}{v_r}=\frac{\sum_{k\in L\cup R} \frac{x_k}{v_k}v_k^2}{\sum_{k\in L\cup R}v_k^2}
< \frac{x_l}{v_l}\frac{\sum_{k\in L\cup R} v_k^2}{\sum_{k\in L\cup R}v_k^2}=\frac{y_l}{v_l}.
\]
We conclude that any inequality in the description of $P$ indexed by $(l,r)\in L^c\times R$ is met by $y$ and strict.
The case $L\times R^c$ is analogous.
The proof is complete.
\end{proof}

Using $y^{L,R}(x)$ as candidate for the projection if $(L,R)=(\bar{L}(x),\bar{R}(x))$ is the naive approach described in Example \ref{exa_projection}.
Example \ref{exa_projection_success} shows that this approach can succeed, and Example \ref{exa_projection_fail} shows that it may fail. 

\begin{exa}\label{exa_projection_success}
Let $n=4$ and $N=2$, with $(x_1,x_2,x_3,x_4)=(1,2,-3,-4)$ and $(v_1,v_2,v_3,v_4)=(1,1,-1,-1)$.
In this case, $\bar{L}(x)=\{1,2\}$ and $\bar{R}(x)=\{3,4\}$, and if we choose $(L,R)=(\bar{L},\bar{R})$, we calculate that $(y_1,y_2,y_3,y_4)=(2.5,2.5,-2.5,-2.5)$.
By Proposition \ref{prop_projection}, $y$ is the projection of $x$ onto $P$.
\end{exa}

\begin{exa}\label{exa_projection_fail}
Let $n=4$ and $N=2$, with $(x_1,x_2,x_3,x_4)=(1,3,-2,-4)$ and $(v_1,v_2,v_3,v_4)=(1,1,-1,-1)$.
Note that as in Example \ref{exa_projection_success}, $\bar{L}(x)=\{1,2\}$ and $\bar{R}(x)=\{3,4\}$ and $(y_1,y_2,y_3,y_4)=(2.5,2.5,-2.5,-2.5)$ if we choose $(L,R)=(\bar{L},\bar{R})$.
Consequentially, it holds that $x-y=(0.5,-1.5,0.5,-1.5)$.
Since this is not an element of $(-\infty,0]^n$, $y$ is not the projection of $x$ onto $P$ by Proposition \ref{prop_projection}.
\end{exa}

\subsubsection*{Algorithm for construction of the projection}
Here, we provide an algorithm that constructs the projection of a vector $x\in\mathbb{R}^n$ onto the polyhedral cone $P:=\spa\{v\}\oplus[0,\infty)^n$.
We will assume that the indices of the fractions $x_j/v_j$ are such that the expressions are in increasing order within $\{1,\ldots, N\}$ and in decreasing order in $\{N+1,\ldots,n\}$:
\begin{equation}\label{eq_increasing_decreasing}
\frac{x_1}{v_1}<\frac{x_2}{v_2}<\ldots,\frac{x_N}{v_N},\quad
\frac{x_{n}}{v_{n}}<\frac{x_{n-1}}{v_{n-1}}<\ldots,\frac{x_{N+1}}{v_{N+1}}
\end{equation}
Our discussion here is entirely deterministic, so this re-indexing occurs without loss of generality.

Algorithm \ref{alg_2} constructs the projection $y$ of $x$ onto $P$.
Some intuition is given in Example \ref{exa_projection}.
Throughout the algorithm, we increase the index sets $L\subset\{1,\ldots,N\}$ and $R\subset\{N+1,\ldots,n\}$.
The sets $\bar{L}:=\bar{L}(x)$ and $\bar{R}:=\bar{R}(x)$ are defined as in \eqref{eq_LR_bar_def}.
For an $x\in\mathbb{R}$ and a pair $(L,R)\subset\mathcal{I}_N$, the vector $y^{LR}:=y^{LR}(x)$ and $\Sigma_{LR}:=\Sigma_{LR}(x)$ are defined as in \eqref{eq_defn_projection}.
The vector $y\in\mathbb{R}^n$ is updated throughout the algorithm, and eventually equal to the projection of $x$ onto $P$.
This claim is later proved in Theorem \ref{thm_alg_2_correct}.
We provide the proof in Appendix \ref{app_proof_correct}.

\begin{algorithm}[tb]
    \caption{Projection of $x$ onto $P$.}\label{alg_2}
    
\begin{algorithmic}[htb]
\STATE{\textbf{Input} $x\in\mathbb{R}^n$, $v_1,\ldots,v_N>0$, $v_{N+1},\ldots,v_n<0$, satisfying \eqref{eq_increasing_decreasing}.}
\STATE{    
    \textbf{Output} $y\in P$, the projection of $x$ onto $P$.}
\STATE{ $y\leftarrow x$.}
\STATE{$L\leftarrow\{1\}$, $R\leftarrow\{N+1\}$}
    \WHILE{$y\not\in P$}
        \STATE{ $y\leftarrow y^{LR}$.}
        \IF{$|L|< |\bar{L}|$ and $x_{|L|+1}/v_{|L|+1}<\Sigma_{LR}$}
            \STATE{$L\leftarrow L\cup\{|L|+1\}$.}
        \ENDIF
        \IF{$|R|<|\bar{R}|$ and $x_{|R|+1}/v_{|R|+1}>\Sigma_{LR}$}
            \STATE{$R\leftarrow R\cup\{|R|+1\}$.}
        \ENDIF
    \ENDWHILE
\end{algorithmic}
\end{algorithm}

\subsection{Intrinsic volumes}\label{sec_intrinsic}
We have found an inequality description for polyhedral cones of the form $P:=\spa\{v\}\oplus [0,\infty)^n$ (Theorem \ref{thm_description}), and know how to construct projections onto $P$ (Algorithm \ref{alg_2}).
We proceed to calculate their intrinsic volumes.
That is, for $k\in\{0,\ldots,n\}$, we calculate the probability that a standard Gaussian vector $g$ in $\mathbb{R}^n$, when projected onto $P$, lands on a face of dimension $k$ (more precisely, the relative interior of a face of dimension $k$):
\[
\nu_k(P):=\mathbb{P}\left[\bigcup_{F\in\mathcal{F}_k(P)}\Pi_P(g)\in\relint F\right].
\]
We emphasize that we treat $v$ as fixed here and not a random variable.
Using the inequality description in Theorem \ref{thm_description} and the construction of the projection in Algorithm \ref{alg_2}, we can give a formula for the intrinsic volumes.

\begin{thm}\label{thm_intrinsic}
Fix $v\in(\mathbb{R}\setminus\{0\})^p$, such that not all entries are positive or negative.
Define $P:=\spa\{v\}\oplus[0,\infty)^n$ and:
\[
\mathcal{I}:=\{l\in[n]:v_l>0\}\times \{r\in[n]:v_r<0\}.
\]
Letting $g\sim\mathcal{N}(0,1)$, $i:=\sqrt{-1}$, as well as $\|v\|_{LR}^2:=\sum_{j\in L\cup R}v_j^2$, it holds that for all $k\in\{1,\ldots,n-1\}$,
\[
\nu_{n-k}(P)
=\sum_{\substack{(L,R)\subset\mathcal{I}\\|L|+|R|=k+1}}
\mathbb{E}\left[\prod_{j\in L\cup R}\Phi\left(\frac{igv_j}{\|v\|_{LR}}\right)\right]\mathbb{E}\left[\prod_{j\in[n]\setminus(L\cup R)}\Phi\left(\frac{gv_j}{\|v\|_{LR}}\right)\right].
\]
Moreover,
\[
\nu_0(P)=0,\quad \nu_n(P)=1-\sum_{k=0}^{n-1}\nu_{k}(P).
\]
\end{thm}

Theorem \ref{thm_intrinsic} uses the analytic continuation of the Gaussian cumulative distribution function:
\[
\Phi:\mathbb{C}\rightarrow\mathbb{C},\quad z\mapsto \frac{1}{2}+\frac{1}{\sqrt{2\pi}}\int_0^ze^{-\frac{t^2}{2}}dt.
\]
Complex numbers are not necessary to express the intrinsic volumes here, one could write the formula with real expressions only (see Remark \ref{rem_real}).
We provide the complex formulation here, as it is more compact. 
Finally, we remark that if all entries of $v$ are strictly positive or negative, then $\spa\{v\}\oplus[0,\infty)^n=\mathbb{R}^n$, and so the $k$-th intrinsic volume is $1$ if $k=n$ and $0$ else.

\begin{exa}
It is instructive to keep the special case in mind, where $v\in\{-1,+1\}^n$.
Defining the number of positive elements $N_L:=\#\{j\in[n]:v_j>0\}$, and the number of negative elements $N_R:=\#\{r\in [n]:v_j<0\}$, and using the abbreviation $\tilde{\Phi}(z):=\Phi(z/\sqrt{k+1})$, the $n-k$-th intrinsic volume reads:
\[
\sum_{l=1}^{k}{N_L\choose l}{N_R\choose r}
\mathbb{E}\left[\tilde{\Phi}\left(ig\right)^l\tilde{\Phi}\left(-ig\right)^{r}
\right]\mathbb{E}\left[\tilde{\Phi}\left(-g\right)^{N_L-l}\tilde{\Phi}\left(g\right)^{N_R-r}\right].
\]
\end{exa}

\begin{rem}\label{rem_combinatorics}
Upon inspection of the proof of Theorem \ref{thm_intrinsic}, we see another characterization of the intrinsic volumes of $P$ that may be instructive.
With the notation of Theorem \ref{thm_intrinsic} and a Gaussian $g\sim\mathcal{N}(0,I_n)$,
\[
\nu_{n-k}(P)=
\sum_{\substack{(L,R)\subset\mathcal{I}\\|L|+|R|=k+1}}\mathbb{P}\left[
\max_{j\in L^c\cup R}\frac{g_j}{v_j}\geq 
\Sigma_{LR}(g)\geq
\max_{k\in L\cup R^c}\frac{g_j}{v_j}
\right].
\]
\end{rem}

\begin{proof}[Proof of Theorem \ref{thm_intrinsic}]
We start with the observation that $\nu_0(P)=0$. 
To see why, note that every point $x\in P$ lies in the linear subspace $x+\spa\{v\}\subset P$.
Therefore, $P$ has no vertices.
The $n$-th intrinsic volume follows, once the cases $k\in\{1,\ldots,n-1\}$ are known.
We use the notation $P:=\spa\{v\}\oplus [0,\infty)^n$.
We know by the correctness of Algorithm \ref{alg_2} (see Theorem \ref{thm_alg_2_correct}), that when the vector $g$ is projected onto $P$, there exists an index pair $(L,R)\subset\mathcal{I}_N$, such that the projection $\Pi_P(g)$ is of the form:
\[
\Sigma_{LR}(g):=\frac{\sum_{k\in L\cup R} g_kv_k}{\sum_{k\in L\cup R}v_k^2},\quad
y_j^{LR}(g):=\begin{cases}
v_j\Sigma_{LR}(g),&\text{if }j\in L\cup R,\\
g_j,&\text{else.}
\end{cases}
\]
From the inequality description in Theorem \ref{thm_description}, we know that the faces of $P$ can be indexed by such index pairs $(L, R)\subset\mathcal{I}_N$ as well.
When all constraints corresponding to a pair $(L, R)\subset\mathcal{I}_N$ are satisfied with equality, for any $j,q\in L\cup R$ it holds that $g_j/v_j=g_q/v_q$.
So, the face corresponding to the pair $(L,R)$ is a $n-(|L|+|R|-1)$ dimensional face.

Necessary and sufficient conditions for $y^{LR}(g)$ to be the projection of $g$ onto $P$ are given in Proposition \ref{prop_projection}.
Let $L^c:=\{l\in[n]:v_l>0\}\setminus L$ and define $R^c:=\{r\in[n]:v_r<0\}\setminus R$.
The first condition is $y^{LR}(g)\in P$, which holds if and only if:
\[
\bigcap_{l\in L^c}\frac{g_l}{v_l}\geq \Sigma_{LR}(g),\quad \bigcap_{r\in R^c} \Sigma_{LR}(g)\geq  \frac{g_r}{v_r}.
\]
The second condition requires $g-y^{LR}(g)\in \ker\{v\}\cap (-\infty,0]^n $, the third requires that $g-y^{LR}(g)\in \ker\{y\}$.
That $g-y^{LR}(g)$ is orthogonal to $v$ and $y$ follows from Lemma \ref{lem_projection_kernels}.
The vector $g-y^{LR}(g)$ is negative if and only if:
\[
\bigcap_{r\in R}\frac{g_r}{v_r}\geq \Sigma_{LR}(g),
\quad
\bigcap_{l\in L} \Sigma_{LR}(g)\geq \frac{g_l}{v_l} .
\]
So, we find that:
\[
\mathbb{P}\left[\Pi_P(g)=y^{LR}(g)\right]
=\mathbb{P}\left[
\bigcap_{\substack{{j\in L^c\cup R}\\{k\in L\cup R^c}}}\frac{g_j}{v_j}\geq \Sigma_{LR}\geq \frac{g_k}{v_k}
\right].
\]
The events (inequalities) indexed by $L\cup R$ are independent of the events indexed by $L^c\cup R^c$, since the $g_i$ follow a Gaussian distribution.
To see why, take $j\in L\cup R$ and $k\in L^c\cup R^c$, and compute:
\[
\mathbb{E}\left[
\left(\frac{g_j}{v_j}-\Sigma_{LR}(g)\right)
\left(\frac{g_k}{v_k}-\Sigma_{LR}(g)\right)
\right]
=-\mathbb{E}\left[\frac{g_j}{v_j}\Sigma_{LR}(g)-\Sigma_{LR}^2(g)
\right]=0.
\]
It follows that:
\[
\mathbb{P}\left[
\bigcap_{\substack{{j\in L^c\cup R}\\{k\in L\cup R^c}}}\frac{g_j}{v_j}\geq \Sigma_{LR}\geq \frac{g_k}{v_k}
\right]
\]
\[
=
\mathbb{P}\left[
\bigcap_{(l,r)\in L^c\times R^c}\frac{g_l}{v_l}\leq \Sigma_{LR}(g)\leq \frac{g_r}{v_r}
\right]
\mathbb{P}\left[
\bigcap_{(l,r)\in L\times R}\frac{g_r}{v_r}\geq \Sigma_{LR}(g)\geq \frac{g_l}{v_l}
\right].
\]
The first term is easier to calculate: as the index sets do not overlap, $\Sigma_{LR}(g)$ is independent of the $g_j/v_j$.
Note that $\Sigma_{LR}(g)$ follows a normal distribution with variance $1/\|v\|^2_{LR}$.
Flipping the signs gives the expression in the theorem.
The second term takes more work. 
Using Corollary \ref{cor_youden_general}, we find:
\[
\mathbb{E}\left[\prod_{l\in L^c}\Phi\left(\frac{-g_1|v_l|}{\|v\|_{LR}}\right)\prod_{r\in R^c}\Phi\left(\frac{g_1|v_r|}{\|v\|_{LR}}\right)\right]
\mathbb{E}\left[\prod_{l\in L}\Phi\left(\frac{ig_1|v_l|}{\|v\|_{LR}}\right)\prod_{r\in R}\Phi\left(\frac{-ig_1|v_r|}{\|v\|_{LR}}\right)\right].
\]
The proof is complete.
\end{proof}

\section*{Acknowledgments} 
The author is grateful for several helpful discussions and comments, in particular, to
Sara van de Geer, 
Satoshi Hayakawa,
Silvan Vollmer,
Guillaume Aubrun,
Michael Law,
Malte Londschien,
Cyrill Scheidegger,
Maybritt Schillinger and 
Christoph Schultheiss.

\bibliographystyle{plainnat}
\bibliography{bibliography}
\newpage

\begin{appendix}

\section{Definitions of linear separability}\label{app_defn}
We introduced the definition of linear separability we use in Section \ref{sec_int_models}.
In this section, we show some other definitions that exist in the literature and show that for our purposes, they are equivalent to the definition we provided in \eqref{eq_intro_sign}.

The discussion in this section is entirely deterministic.
So, we do not distinguish between the case with and without intercept.
The intercept can be symbolically removed from the definitions, using the standard trick: redefining $\tilde{x}_j:=(1,x_j^T)^T$ and $\tilde{\beta}:=(\beta_0,\beta^T)^T$, or $\tilde{x}_j:=(x_j^T,1)^T$ and $\tilde{\beta}:=(\beta^T,\beta_0)^T$, if preferred.
Moreover, we will use the notation $z_j:=y_jx_j$.

The definition that stands out the most is possibly \eqref{eq_separation_minkowski}.
It stems from the characterization in Proposition \ref{prop_candes}, due to \cite{candes2020phase}.
Here, we formulate it with our setting in mind, where we either have a signal but not intercept, so for $\beta^*=e_1$, $y_jx_j\sim y_jx_{j,1}+x_{j,2:p}$, or without signal but with intercept, $y_j(1,x_j^T)^T\sim y_j+x_{j}$.
To include both an intercept and signal, one could state the condition \eqref{eq_separation_minkowski} with $\spa\{z_{\cdot,1}\}\oplus \spa\{z_{\cdot,2}\}\oplus[0,\infty)^n$, and prove a similar equivalence.

\begin{defn}
We say that the data are \textit{completely separable} if there exists a $\beta\in\mathbb{R}^p$, such that for all $i\in[n]$, $z_i^T\beta>0$.
Symbolically:
\begin{equation}\label{eq_separation_strong}
\exists \beta\in \mathbb{R}^p,\forall i\in[n],z_i^T\beta>0
\end{equation}
We say that the data are \textit{weakly separable} if there exists a $\beta\in\mathbb{R}^p$, such that for all $i\in[n]$, $z_i^T\beta\geq 0$, and at least one inequality is strict.
Formally:
\begin{equation}\label{eq_separation_weak}
\exists \beta\in \mathbb{R}^p,\left(\forall i\in[n],z_i^T\beta\geq 0\right)\wedge \left(\exists i\in[n], z_i^T\beta\neq 0\right)
\end{equation}
We say that the data are \textit{non-trivially separable} if there exists a non-zero $\beta\in\mathbb{R}^p$, such that for all $i\in[n]$, $z_i^T\beta\geq 0$.
In other words:
\begin{equation}\label{eq_separation_nonzero}
\exists \beta\in \mathbb{R}^p\setminus\{0\},\forall i\in[n],z_i^T\beta\geq 0
\end{equation}
We say that the data are \textit{Candès-Sur separable}, if:
\begin{equation}\label{eq_separation_minkowski}
\spa\{z_{\cdot,2},\ldots z_{\cdot,p}\}\cap \left(\spa\{z_{\cdot,1}\}\oplus[0,\infty)^n\right)\neq \{0\}.
\end{equation}
\end{defn}

Some authors also use the definition \textit{quasi-complete separation}, which is the event \eqref{eq_separation_weak}$\setminus$\eqref{eq_separation_strong}, see  \cite{albert1984existence}.
Proposition \ref{prop_equivalence} describes in what sense they are equivalent.
Under no additional assumptions, the following implications hold:
\[
\eqref{eq_separation_strong}\Rightarrow\eqref{eq_separation_weak}\Rightarrow \eqref{eq_separation_nonzero}\Leftarrow\eqref{eq_separation_minkowski}
\]
Remark \ref{rem_equivalence} below shows that for our purposes, all definitions are equivalent.

\begin{prop}\label{prop_equivalence}
\eqref{eq_separation_strong} implies \eqref{eq_separation_weak}, \eqref{eq_separation_weak} implies \eqref{eq_separation_nonzero}.
Neither implication is an equivalence in general.
If the vectors $z_1,\ldots,z_n$ span $\mathbb{R}^p$, then \eqref{eq_separation_nonzero} implies \eqref{eq_separation_weak}.
If any $p$ vectors of $z_1,\ldots,z_n$ are linearly independent, then \eqref{eq_separation_weak} implies \eqref{eq_separation_strong}.

Moreover, \eqref{eq_separation_strong} does not imply \eqref{eq_separation_minkowski}.
\eqref{eq_separation_minkowski} does not imply \eqref{eq_separation_weak}.
\eqref{eq_separation_minkowski} implies \eqref{eq_separation_nonzero}.
If $\spa\{z_{\cdot,2},\ldots z_{\cdot,p}\}\neq \{0\}$, then \eqref{eq_separation_strong} implies \eqref{eq_separation_minkowski}.
\end{prop}

\begin{proof}[Proof of Proposition \ref{prop_equivalence}]
Clearly, \eqref{eq_separation_strong} implies \eqref{eq_separation_weak} and \eqref{eq_separation_weak} implies \eqref{eq_separation_nonzero}.
However, the definitions are in general not equivalent.
To see that \eqref{eq_separation_weak} does not imply \eqref{eq_separation_strong}, take $n=2$, $p>1$, $z_1=-e_1$, $z_2=e_1$ and $z_3=e_2$.
Then, \eqref{eq_separation_weak} occurs, since for example $\beta=e_2$ is a solution.
However, as any solution must satisfy $\beta_1=0$, the event \eqref{eq_separation_strong} does not occur.
To see that \eqref{eq_separation_nonzero} does not imply \eqref{eq_separation_weak}, take $n=2$, $p>1$, $z_1=-e_1$ and $z_2=e_1$. 
The event in \eqref{eq_separation_nonzero} occurs, since any $\beta\in S^{p-1}$ with $\beta_1=0$ is a solution. However, \eqref{eq_separation_weak} does not occur.

We show that if $z_1,\ldots,z_n$ span $\mathbb{R}^p$, then \eqref{eq_separation_nonzero} implies \eqref{eq_separation_weak}.
To see why, note that if \eqref{eq_separation_nonzero} holds, there is a nonzero $\beta$ with $z_i^T\beta\geq0$ for all $i\in[n]$.
To reach a contradiction, suppose that all inequalities were in fact equalities.
Then $\beta$ would lie in the orthogonal complement of the span of $z_1,\ldots,z_n$.
Since $z_1,\ldots,z_n$ span $\mathbb{R}^p$, this implies that $\beta=0$, which is a contradiction.

We show that if any $p$ vectors of $z_1,\ldots,z_n$ are linearly independent, then \eqref{eq_separation_weak} implies \eqref{eq_separation_strong}.
We proceed by contradiction.
Suppose that \eqref{eq_separation_weak} holds but not \eqref{eq_separation_strong}.
Let $C\subset\mathbb{R}^p$ be the set of solutions of \eqref{eq_separation_weak}, which is nonempty by assumption.
Define the $n\times p$ matrix $Z$, whose rows are the vectors $z_1,\ldots, z_n$.
If for all $i\in[n]$ there exists a $\beta^j\in C$ and a $j\in[p]$ such that $(z_j^T\beta^i)_i>0$, then the sum $\sum_{i=1}^n\beta^i$ leads to a solution of \eqref{eq_separation_strong}.
If not, then there exists an $i$ such that for all $\beta \in C$ and all $j\in[p]$ $(z_j^T\beta^i)_i\leq 0$.
Consequentially, the collection of inequalities $\{z_j^T\beta\geq0:j\in[n]\}$ implies that $z_i^T\beta\leq0$.
So, there exists a $\lambda\in [0,\infty)^{n-1}$, such that $\sum_{j\neq i}\lambda_j z_j=-z_i$.
But then, the origin lies in the convex hull of the vectors $z_1,\ldots,z_n$.
Since \eqref{eq_separation_weak} holds, the origin cannot lie in the interior of $Conv(z_1,\ldots,z_n)$, so it must lie on a face of dimension $k\leq p-1$.
As this face is itself the convex hull of a collection of points in $z_1,\ldots,z_n$, by Carath{\'e}odory's theorem there exists a subset of at most $k+1\leq p$ of these points such that the origin is a convex combination of these points.
Consequentially, there exist $p$ vectors of $z_1,\ldots,z_n$ which are not linearly independent.
Contradiction.

To show \eqref{eq_separation_strong} does not imply \eqref{eq_separation_minkowski}, we take $z_{\cdot,2},\ldots, z_{\cdot,p}=0$ and $z_{\cdot,1}=(1,\ldots,1)$.
Then, $e_1$ is a solution to \eqref{eq_separation_strong}, but \eqref{eq_separation_minkowski} does not hold.

Now suppose that $\spa\{z_{\cdot,2},\ldots z_{\cdot,p}\}\neq \{0\}$ and \eqref{eq_separation_strong} holds.
Let $Z_{-1}$ be the $n\times p-1$-matrix with columns $x_{\cdot,j}$ for $j\in\{2,\ldots,p\}$.
There exist $\alpha\in\mathbb{R}^{p-1},\beta\in\mathbb{R}$ and $u\in(0,\infty)^n$ such that:
\begin{equation}\label{eq_separation_minkowski_equation}
Z_{-1}\alpha=u+z_{\cdot,1}\beta
\end{equation}
If $u\neq z_{\cdot,1}\beta$, \eqref{eq_separation_minkowski} holds.
Else, let $\lambda'>0$ be such that for all $i\in[n]$, it holds that $\lambda'|(Z_{-1}\alpha)_i|\leq u_i/2$.
Therefore, for any $i\in[n]$,
\[
u_i':=\lambda'(Z_{-1}\alpha)_i-z_{i,1}\beta\geq -\frac{u_i}{2}+u_i>0
\]
So, $Z_{-1}(\lambda'\alpha)=z_{\cdot,1}\beta+u'$, and therefore \eqref{eq_separation_minkowski} holds.

Suppose that \eqref{eq_separation_minkowski} holds.
We show that \eqref{eq_separation_nonzero} holds.
There exists a nonzero element in $\spa\{z_{\cdot,2},\ldots z_{\cdot,p}\}\cap \left(\spa\{z_{\cdot,1}\}\oplus[0,\infty)^n\right)$, meaning that there exist $\alpha\in\mathbb{R}^{p-1}$, $\beta\in\mathbb{R}$ and an $u\in[0,\infty)^n$, such that \eqref{eq_separation_minkowski_equation} holds, where both sides of the equation are not zero.
From this, it follows that the vector $(\beta,\alpha)$ is nonzero.
Therefore, \eqref{eq_separation_nonzero} holds.

We show that \eqref{eq_separation_minkowski} does not imply \eqref{eq_separation_weak}.
Let $n=p=2$ and $z_{\cdot ,1}=z_{\cdot, 2}=e_1-e_2$.
Then, \eqref{eq_separation_minkowski} holds if there exist $\alpha,\beta\in\mathbb{R}^p$ and $u\in[0,\infty)^n$, such that:
\[
\begin{pmatrix} 1\\-1\end{pmatrix}\begin{pmatrix} 1&0\\0&0\end{pmatrix}\alpha
=\begin{pmatrix} -1\\1\end{pmatrix}\begin{pmatrix} 0&0\\0&1\end{pmatrix}\beta
+u
\]
All solutions satisfy $\alpha_1=-\beta_2$ and $u=0$.
So, \eqref{eq_separation_minkowski} holds and \eqref{eq_separation_weak} does not.
The proof is complete.
\end{proof}

\begin{rem}\label{rem_equivalence}
Let $yX$ be the $n\times p$ matrix with rows $y_jx_j$.
Suppose that the first column is nonzero and that the distribution of the remaining column vectors of $yX$, conditional on all other columns, is absolutely continuous with respect to the Lebesgue measure.
Then, any $p$ vectors in $\{y_1x_1,\ldots,y_nx_n\}$ span $\mathbb{R}^p$ almost surely, and hence \eqref{eq_separation_strong} is equivalent to \eqref{eq_separation_minkowski} almost surely by Proposition \ref{prop_equivalence}.
Note that this includes the case with intercept, as we may take the first column of $yX$ to be $y_1,\ldots,y_n$.

A standard induction argument shows why.
Without loss of generality take $p=n$.
Then, the rows of $yX$ span $\mathbb{R}^p$ if and only if the columns of $yX$ are linearly independent.
For any $j\in\{1,\ldots,p-1\}$, the span of the first $j$ columns of $yX$ has Lebesgue measure zero. 
As the $j+1$-th column of $yX$ has a distribution absolutely continuous with respect to the Lebesgue measure conditional on the previous $j$ columns, the probability that it lies in the span of the previous $j$ is zero.
It follows that the columns of $yX$ are linearly independent almost surely.
\end{rem}

\section{A formula for Youden's demon problem}\label{sec_youden}

Let $m\in\{2,3\ldots\}$, $k\in\{1,\ldots,m-1\}$.
Let $g\sim\mathcal{N}(0,I_m)$, $v\in(\mathbb{R}\setminus\{0\})^m$ and:
\[
\Sigma(v):=\sum_{j\in [m]}\frac{g_jv_j}{\|v\|_2^2},\quad
\|v\|_2^2:=\sum_{j\in [m]}v_j^2.
\]
In this section, we are interested in calculating the following probability:
\begin{equation}\label{eq_youden_general}
\mathbb{P}\left[
\bigcap_{\substack{1\leq l\leq k\\ k+1\leq r\leq m }}\frac{g_l}{v_l}\leq \Sigma(v)\leq \frac{g_r}{v_r}
\right].
\end{equation}
Our motivation for studying this probability comes from the calculation of the intrinsic volumes of the polyhedral cone $\spa\{v\}\oplus [0,\infty)^n$.
A version of this problem has already been studied in the literature:
If we take $v=(1,\ldots,1)$,
\begin{equation}\label{eq_youden}
\mathbb{P}\left[\max_{1\leq l\leq k}g_l\leq  \frac{1}{m}\sum_{j=1}^mg_j\leq \min_{k+1\leq r\leq m} g_r\right].
\end{equation}
In other words: What is the probability that the sample mean is larger than the first $k$ observations and smaller than the rest?
This is known as \textit{Youden's demon problem} \cite{kendall1954two}.

So, the probability in \eqref{eq_youden_general} can be viewed as a weighted version of Youden's demon problem in the Gaussian case.
It is worth pointing out that \eqref{eq_youden_general} is not Youden's demon problem for independent Gaussians with arbitrary standard deviations $1/|v_j|$, due to the way the weighted average $\Sigma_{LR}$ is calculated.

We study the probability in \eqref{eq_youden_general}.
We first recall some results available in the literature about Youden's demon problem in Section \ref{sec_youden_hist}.
We then provide the main results in Section \ref{sec_youden_main}.
They follow from an adaptation of a proof by \cite{kabluchko2020absorption}.
We first provide a sketch in Section \ref{sec_youden_sketch}, and in Section \ref{sec_youden_proof} the proof itself.

\subsection{Youden's demon problem}\label{sec_youden_hist}
Youden's demon problem was originally formulated in \cite{kendall1954two}.
It asks:\textit{ What is the probability that in a sample of $m$ independent observations, the sample mean lies between the $k$-th and $k+1$-th order statistic?}
Equivalently, one could calculate \eqref{eq_youden}, the probability that the sample mean lies between the $k$-th and $k+1$-th observation, with a factor of ${m\choose k}$.

For the Gaussian case, the values for $m\in\{3,\ldots,10\}$ and $k=m-1$ were computed in \citep{kendall1954two}.
\cite{david1962sample} derived asymptotic approximations for the Gaussian case.
For the Gaussian case, \cite{david1963sample} established an identity relating Youden's demon problem to the spherical measure of equilateral spherical simplices.
They also provided approximations for these spherical measures. 
\cite{six1983recurrence} later showed that the problem can be solved with a recurrence relation.
\cite{dmitrienko1997demon} calculated the probability for the uniform distribution, and \cite{dmitrienko1998demon} for the exponential distribution.
Here, we extend these results by providing a closed-form solution for the Gaussian case.

\subsection{Statement of the results}\label{sec_youden_main}
Here, we provide the main result of this section, Theorem \ref{thm_kabluchko}, as well as two Corollaries \ref{cor_youden_general} and \ref{cor_youden}.
Theorem \ref{thm_kabluchko} is a slight generalization of \cite[Proposition 1.4]{kabluchko2020absorption}.
In its proof, we follow the same line of reasoning.
We provide a sketch of the idea below.
Below, we define the complex Gaussian cumulative distribution function:
\begin{equation}\label{eq_defn_phi}
\Phi:\mathbb{C}\rightarrow\mathbb{C}\quad z\mapsto \frac{1}{2}+\frac{1}{\sqrt{2\pi}}\int_0^{z}e^{-\frac{t^2}{2}}dt.
\end{equation}

\begin{thm}\label{thm_kabluchko}
Let $m\in\{2,3,\ldots\}$, $k\in\{1,\ldots,m\}$, fix $v\in(\mathbb{R}\setminus\{0\})^m$ and $\rho\in[-1,\infty)$.
Suppose that $\eta_j^\rho$ are mean zero Gaussians, such that for all $l,r\in\{1,\ldots,m\}$:
\begin{equation}\label{eq_defn_eta}
\mathbb{E}[\eta_l^\rho\eta_r^\rho]=\frac{1}{v^2_l}1\{l=r\}+\rho\frac{1}{\|v\|_2^2}.
\end{equation}
Then, with $g\sim\mathcal{N}(0,1)$, it holds that:
\begin{equation}\label{eq_youden_general_rho}
\mathbb{P}\left[\bigcap_{j=1}^k \eta_j^\rho<0\cap\bigcap_{j=k+1}^{m} \eta_j^\rho>0\right]
=\mathbb{E}\prod_{l=1}^k\Phi\left(\frac{\sqrt{\rho}g|v_l|}{\|v\|_2}\right)\prod_{r=k+1}^m\Phi\left(\frac{-\sqrt{\rho}g|v_r|}{\|v\|_2}\right).
\end{equation}
\end{thm}

The case $v\equiv 1$ and $k=m$ follows directly from \cite[Proposition 1.4]{kabluchko2020absorption}.
We are particularly interested in the following version of Theorem \ref{thm_kabluchko}, concerning the case $\rho=-1$.

\begin{cor}\label{cor_youden_general}
Let $m\in\{2,3,\ldots\}$, $k\in\{1,\ldots,m\}$, fix $v\in(\mathbb{R}\setminus\{0\})^m$ and $g\sim\mathcal{N}(0,I_m)$.
It holds that:
\[
\mathbb{P}\left[\bigcap_{\substack{1\leq l\leq k\\ k+1\leq r\leq m }}\frac{g_l}{v_l}\leq \Sigma(g)\leq \frac{g_r}{v_r}\right]
=\mathbb{E}\prod_{l=1}^k\Phi\left(\frac{ig_1|v_l|}{\|v\|_2}\right)\prod_{r=k+1}^m\Phi\left(\frac{-ig_1|v_r|}{\|v\|_2}\right).
\]
Here, $i:=\sqrt{-1}$ is the imaginary number.
\end{cor}

We also mention a nice special case, which occurs if $v\equiv 1$.
It provides a closed-form solution to Youden's demon problem for Gaussian covariates.

\begin{cor}\label{cor_youden}
Let $m\in\{2,3,\ldots\}$, $k\in\{1,\ldots,m\}$ and $g\sim\mathcal{N}(0,I_m)$.
Then:
\[
\mathbb{P}\left[
\max_{1\leq l\leq k}g_l\leq \frac{1}{m}\sum_{j=1}^mg_j
\leq \min_{k+1\leq r\leq m}g_r
\right]
=\mathbb{E}\Phi\left(\frac{ig_1}{\sqrt{m}}\right)^k\Phi\left(\frac{-ig_1}{\sqrt{m}}\right)^{m-k}.
\]
\end{cor} 

\begin{rem}\label{rem_real}
Here, we remark that the expectations in Corollaries \eqref{cor_youden_general} and \eqref{cor_youden} are real, and can be expressed without the imaginary number $i$.
To see why, note that for any $k\in\mathbb{N}$ and $g\sim\mathcal{N}(0,1)$,
\[
\mathbb{E}\prod_{j=1}^{k}\int_0^{\frac{ig|v_j|}{\|v\|_2}}e^{-\frac{t^2}{2}}dt
=(-1)^k\mathbb{E}\prod_{j=1}^{k}\int_0^{\frac{-ig|v_j|}{\|v\|_2}}e^{-\frac{t^2}{2}}dt.
\]
As $g$ and $-g$ have the same distribution, this expression vanishes whenever $k$ is odd.
Consequentially:
\[
\frac{1}{2^m}\mathbb{E}\prod_{l=1}^k\left(1+2i\int_0^{\frac{g|v_l|}{\|v\|_2}}e^{\frac{t^2}{2}}dt\right)
\prod_{r=k+1}^m\left(1-2i\int_0^{\frac{g|v_r|}{\|v\|_2}}e^{\frac{t^2}{2}}dt\right)
\]
\[
=\frac{1}{2^m}\sum_{S\subset [m]}\sigma_{S}2^{|S|}\mathbb{E}\prod_{j\in S}\int_0^{\frac{g|v_j|}{\|v\|_2}}e^{\frac{t^2}{2}}dt,
\]
where:
\[
\sigma_S:=\begin{cases}
(-1)^{|S|/2}(-1)^{|S\cap \{k+1,\ldots,m\}|},&\text{if }|S|\text{ even},\\
0,&\text{else.}
\end{cases}
\]
In the special case where $v_j=1$ for all $j\in\{1,\ldots,m\}$, this gives:
\[
\mathbb{E}\Phi\left(\frac{ig}{\sqrt{m}}\right)^{k}\Phi\left(-\frac{ig}{\sqrt{m}}\right)^{m-k}
\]
\[
=\frac{1}{2^{m}}\sum_{\substack{1\leq l\leq k\\1\leq r\leq m-k\\ l+r\text{ even}}}{k \choose l}{m-k\choose r}(-1)^r\mathbb{E}\left(-4\int_{0}^{g/\sqrt{m}}e^{t^2/2}dt\right)^{\frac{l+r}{2}}.
\]
\end{rem}

\subsection{Sketch of proof}\label{sec_youden_sketch}
Here, we provide a sketch of the proof of Theorem \ref{thm_kabluchko}.
The case $\rho\geq 0$ is straightforward.
If $\xi_0,\ldots,\xi_m$ are independent standard Gaussians, then the vector with entries $\xi_k/v_k-\sqrt{\rho}\xi_0/\|v\|_2$ has the same distribution as $\eta^\rho$.
We already obtain the identity \eqref{eq_youden_general_rho} for $\rho\geq 0$ from this.
Unfortunately, this does not solve our problem, as we are interested in negative values, particularly $\rho=-1$.
However, it turns out that both sides of the equation are holomorphic functions, from which the equality follows.
Below, $\Re(z)$ and $\Im(z)$ denote the real and imaginary parts of a complex number $z\in\mathbb{C}$.

As in the proof of \cite[Proposition 1.4]{kabluchko2020absorption}, we proceed as follows:
\begin{enumerate}
    \item The identity holds for $\rho\geq 0$.
    \item Write the left-hand-side of \eqref{eq_youden_general_rho} in integral form.
    \item Extend this integral to a function on $D:=\{z\in\mathbb{C}:\Re(z)>-1\}$.
    \item Note that this function is holomorphic in $D$.
    \item Extend the right-hand-side of \eqref{eq_youden_general_rho} to a complex function. 
    \item Show that this function is holomorphic in $D':=\{z\in\mathbb{C}:\Re(z)^2+1>\Im(z)^2\}$.
    \item Argue the case $\rho=-1$ by continuity.
\end{enumerate}

Step 1 shows that \eqref{eq_youden_general_rho} holds on a subset of $\mathbb{C}$ with a limit point.
We show that both sides are holomorphic functions of $z$ on $D\cap D'$, which contains the real interval $(-1,\infty)$.
So, steps 1-6 establish the identity for $\rho>-1$.
Step 7 completes the proof by continuity.

\subsection{Proof of the result}\label{sec_youden_proof}
Here, we provide the proof of Theorem \ref{thm_kabluchko}.
The proof proceeds just as \cite{kabluchko2020absorption}, with slightly more general expressions.
We provide it regardless, for the sake of completeness.
The reader may find a sketch of the proof in Section \ref{sec_youden_sketch}.

Before we start, we establish some useful notation.
Let $\eta^\rho$ be a mean Gaussian vector with the covariance structure given in \eqref{eq_defn_eta}.
We define:
\[
g_{m,k}:\left[-1,\infty\right)\rightarrow\mathbb{R},\quad
\rho\mapsto \mathbb{P}\left[\bigcap_{j=1}^k \eta_j^\rho<0\cap\bigcap_{j=k+1}^{m+1} \eta_j^\rho>0\right].
\]
We will show that $g_{m,k}$ agrees with the function $h_{m,k}:\left[-1,\infty\right)\rightarrow\mathbb{R}$, mapping:
\[
\rho\mapsto \int_{-\infty}^{\infty}\prod_{j=1}^k\Phi\left(\frac{\sqrt{\rho}x|v_j|}{\|v\|_2}\right)\prod_{q=k+1}^m\Phi\left(\frac{-\sqrt{\rho}x|v_q|}{\|v\|_2}\right)\frac{e^{-\frac{x^2}{2}}}{\sqrt{2\pi}}dx.
\]
Negative values of $\rho$ lead to the imaginary number $\sqrt{\rho}=i\sqrt{|\rho|}$.
So, in this notation, Theorem \ref{thm_kabluchko} establishes that for $r\in[-1,\infty)$,
\begin{equation}\label{eq_g=h}
g_{m,k}(\rho)=h_{m,k}(\rho).
\end{equation}

\begin{proof}[{Proof of Theorem \ref{thm_kabluchko}}]
In the \textbf{first step}, we prove \eqref{eq_g=h} for $\rho\geq 0$.
We define independent standard Gaussians $\xi_j$ for $j\in\{0,\ldots,m\}$.
Note that the vector with entries $\xi_k/v_k-\sqrt{\rho}\xi_0/\|v\|_2$ has the same distribution as $\eta^\rho$.
Therefore:
\[
g_{m,k}(\rho)=\mathbb{P}\left[\bigcap_{l=1}^k \frac{\xi_l}{v_l}<\sqrt{\rho}\frac{\xi_0}{\|v\|_2}\cap\bigcap_{r=k+1}^{m} \frac{\xi_r}{v_r}>\sqrt{\rho}\frac{\xi_0}{\|v\|_2}\right]
=h_{m,k}(\rho).
\]

Moving to the \textbf{second step}, we write $g_{m,k}$ in integral form for $r>-1$.
Let $V$ be the $m\times m$ diagonal matrix with entries $V_{jj}:=v_j^2$ and $\rho_v:=\rho/\|v\|_2^2$.
Then, $\Sigma_{m,\rho}=V^{-1}+\rho_v 1_m1_m^T$ is the covariance matrix of $\eta^\rho$, where $1_m$ is the all-ones vector in $\mathbb{R}^m$. 
This matrix is positive definite for $\rho>-1$.
To see why, note that:
\[
\frac{\det(\Sigma_{m,\rho})}{\det(V^{-1})}
=\det\left(I_m+\rho_v V1_m1_m^T\right)
=\det\left(I_1+\rho_v 1_m^TV1_m\right).
\]
The second identity is the Weinstein–Aronszajn identity.
Since all $v_j$ are nonzero, $\det(V^{-1})>0$.
The determinant on the right-hand side is equal to $1+\rho>0$.
So, $\Sigma_{m,\rho}$ is invertible, and we find its inverse with the Sherman-Morrison formula:
\[
(V^{-1}+\rho_v 1_m1_m^T)^{-1}=V-\frac{\rho_vV1_m1_m^TV}{1+\rho_v1_m^TV1_m}=V-\frac{\rho_vv^2(v^2)^{T}}{1+\rho},
\]
where we use $v^2$ to denote the vector $(v_1^2,\ldots,v_m^2)$.
Therefore, defining the orthant $A_{m,k}:=(-\infty,0)^k\times(0,\infty)^{m-k}$, using the multivariate Gaussian cumulative distribution function, we can see that $g_{m,k}(\rho)$ can be written as:
\begin{equation}\label{eq_g_as_integral}
\frac{\prod_{j=1}^m|v_j|}{(2\pi)^{m/2}\sqrt{1+\rho}}\int_{A_{m,k}}\exp\left(\frac{\rho}{2(1+\rho)}\left(\sum_{j=1}^m \frac{v_j^2}{\|v\|_2}x_j\right)^2-\frac{1}{2}\sum_{j=1}^m(v_jx_j)^2\right)dx.
\end{equation}

In \textbf{step three}, we extend this to a complex function.
To emphasize this, we distinguish between the real function $g_{m,k}$, and the complex function $g_{m,k}^\mathbb{C}$.
Let $D:=\left\{z\in\mathbb{C}:\Re(z)>-1\right\}$ and define:
\[
g^\mathbb{C}_{m,k}:D\rightarrow \mathbb{C},\quad \rho\mapsto \eqref{eq_g_as_integral},
\]
where for a complex number $z$, we use $\Re(z)$ and $\Im(z)$ to denote its real and imaginary parts.
We verify that $g^{\mathbb{C}}_{m,k}$ is well-defined on $D$, i.e. that the integral converges.
Since $|e^{z}|=e^{\Re(z)}$, it suffices to consider the real part.
By the Cauchy-Schwarz inequality:
\[
\sum_{j=1}^m\frac{v_j}{\|v\|_2}(v_jx_j)\leq\sqrt{ \sum_{j=1}^m(v_jx_j)^2},
\]
and so the real-valued part in the exponential function can be bounded as follows:
\[
\Re\left(\frac{z}{2(1+z)}\right)\left(\sum_{j=1}^m \frac{v_j^2}{\|v\|_2}x_j\right)^2-\frac{1}{2}\sum_{j=1}^m(v_jx_j)^2
\]
\[
\leq 
\frac{1}{2}\left(\Re \left(\frac{z}{1+z}\right)-1\right)\left(\sum_{j=1}^m \frac{v_j^2}{\|v\|_2}x_j\right)^2.
\]
This expression is negative if $\Re (z)>-1$.
To see why, note that:
\[
1-\Re\left(\frac{z}{1+z}\right)
=
\Re\left(\frac{1+\bar{z}}{(1+z)(1+\bar{z})}\right)=\frac{1+\Re(z)}{(1+\Re(z))^2+\Im(z)^2},
\]
where $\bar{z}:=\Re(z)-i\Im(z)$ denotes the complex conjugate.
It follows that the integral converges on $D$.

In the \textbf{fourth step}, we verify that $g^{\mathbb{C}}_{m,k}$ is holomorphic in $D$. 
The mapping $z\mapsto \sqrt{1+z}$ is holomorphic in $D$.
Hence, so is $z\mapsto 1/\sqrt{1+z}$.
The product of holomorphic functions is holomorphic, so it suffices to prove that the integral in \eqref{eq_g_as_integral} is holomorphic as a function of $\rho$ on $D$.
We do this by use of Morera's Theorem (see e.g. \cite[Theorem 10.17]{rudin}).
So, we verify that it is continuous and its integral along any triangle $\partial \Delta$ vanishes.
Continuity follows from Lebesgue's dominated convergence theorem.
Next, take any triangle $\Delta\subset D$.
We use the notation $g^{\prime\mathbb{C}}_{m,k,z}(x)$ to denote the integrand in the integral $g^{\mathbb{C}}_{m,k}(z)$. 
By Fubini's Theorem:
\[
\oint_{\partial\Delta} g^{\mathbb{C}}_{m,k}(z)dz
=\int_{A_{m,k}}\oint_{\partial\Delta} g^{\prime\mathbb{C}}_{m,k,z}(x)dz dx=0
\]
The inner integral vanishes by Cauchy's Theorem (see e.g. \cite[Theorem 10.13]{rudin}), as the mapping $z\mapsto g^{\prime\mathbb{C}}_{m,k,z}(x)$ is holomorphic in $D$.
We conclude that $g^{\mathbb{C}}_{m,k}$ is holomorphic in $D$.

We proceed to \textbf{step five}, where we extend $h_{m,k}$ to a complex function.
We will later prove that this function is holomorphic.
The function $\Phi$ defined in \eqref{eq_defn_phi} is holomorphic on $\mathbb{C}$.
However, the square root is not.
We denote the integrand of $h_{m,k}(\rho)$ with respect to the Gaussian measure by:
\[
h_{m,k}'(g,z):=\prod_{l=1}^k\Phi\left(\frac{zg|v_l|}{\|v\|_2}\right)\prod_{r=k+1}^m\Phi\left(\frac{-zg|v_r|}{\|v\|_2}\right).
\]
Therefore:
\[
h_{m,k}(\rho)=\int_{-\infty}^{\infty}h_{m,k}'(g,\sqrt{\rho})\frac{e^{-\frac{g^2}{2}}dg}{\sqrt{2\pi}}.
\]
Let $D':=\{z\in\mathbb{C}:\Re(z)^2+1> \Im(z)^2\}$.
Now, we define for each $g\in\mathbb{R}$,
\[
\phi_{m,k,g}:\mathbb{C}\rightarrow\mathbb{C},\quad
z\mapsto h_{m,k}'(g,z)+h_{m,k}'(g,-z),
\]
and:
\[
h^{\mathbb{C}}_{m,k}:D'\rightarrow\mathbb{C},\quad z\mapsto \int_0^\infty \phi_{m,k,g}(\sqrt{z})\frac{e^{-\frac{g^2}{2}}dg}{\sqrt{2\pi}}.
\]
Note that for all real $\rho>-1$, $h_{m,k}(\rho)=h^{\mathbb{C}}_{m,k}(\rho)$.
For any $x\in\mathbb{R}$, the function $z\mapsto \phi_{m,k,g}(z)$ is holomorphic on $\mathbb{C}$.
Moreover, the function is symmetric around zero: for all $z\in\mathbb{C}$, it holds that $\phi_{m,k,g}(z)=\phi_{m,k,g}(-z)$.
So, the function $\phi_{m,k,g}$ can be written as an everywhere converging power series, with only even powers of $z$.
Consequentially, for any $g\in\mathbb{R}$ there exists a sequence $\alpha_j\in\mathbb{C}$, such that for any $z\in\mathbb{C}$,
\[
\phi_{m,k,g}(\sqrt{z})=\sum_{j=0}^{\infty}\alpha_j\sqrt{z}^{2j}=\sum_{j=0}^{\infty}\alpha_jz^{j}
\]
So, the mapping $z\mapsto \phi_{m,k,g}(\sqrt{z})$ is holomorphic on $\mathbb{C}$ as well.

Next, in \textbf{step six} we show that $z\mapsto h_{m,k}^\mathbb{C}(z)$ is holomorphic in $D'$ by use of Morera's Theorem.
So, we verify that it is continuous and its integral along any triangle $\partial\Delta$ in $D'$ vanishes.
The latter is established with Fubini's Theorem and Cauchy's Theorem, as above.
Continuity follows from Lebesgue's dominated convergence theorem, which can be applied thanks to Lemma \ref{lem_youden_dominated}.
We conclude that $h_{m,k}^\mathbb{C}$ is holomorphic in $D'$.

To conclude that $g_{m,k}(\rho)$ and $h_{m,k}(\rho)$ agree for $\rho\in(-1,0]$, we argue as follows.
Take any $\rho\geq 0$.
We showed that:
\[
g_{m,k}^{\mathbb{C}}(\rho)=g_{m,k}(\rho)=h_{m,k}(\rho)=h^{\mathbb{C}}_{m,k}(\rho)
\]
So, the functions $g_{m,k}^{\mathbb{C}}$ and $h^{\mathbb{C}}_{m,k}$ agree on the positive real line, which is a subset of $D\cap D'$ with a limit point.
Consequentially, they agree on $D\cap D'$, by the identity theorem for holomorphic functions.
In particular, for real $\rho>-1$, it holds that $g_{m,k}(\rho)=h_{m,k}(\rho)$.

Finally, in \textbf{step seven}, we argue that the case $\rho=-1$ follows by continuity.
First, we show that the mapping $g_{m,k}:[-1,\infty)\rightarrow [0,1]$ is continuous at $-1$.
So, we claim that as $\rho\searrow -1$, it holds that:
\[
g_{m,k}(\rho)=\mathbb{P}\left[\eta^\rho\in A_{m,k}\right]\rightarrow \mathbb{P}\left[\eta^{-1}\in A_{m,k}\right]
\]
As $A_{m,k}$ is a Borel set, this holds if $\eta^\rho$ converges in distribution to $\eta^{-1}$.
By the Portmanteau Theorem, this convergence in distribution occurs if for any bounded Lipschitz function $f$, the expectation $\mathbb{E}f(\eta^\rho)$ converges to $\mathbb{E}f(\eta^{-1})$.
To show this, we restrict the domain of $\rho$ to $[-1,0]$ and parametrize $\eta^\rho$ as follows: 
We define $\lambda:=1-\sqrt{1+\rho}\in[0,1]$, i.e. $\rho=\lambda^2-2\lambda$ and let $\xi$ be a standard Gaussian in $\mathbb{R}^m$.
Then, the random vector $w_\lambda$ in $\mathbb{R}^m$ with entries $\xi_j/v_j-\lambda\sum_{q=1}^m\xi_qv_q/\|v\|_2^2$ has the same distribution as $\eta^{\rho}$ for the choice of $\lambda$ corresponding to $\rho$.
Now, let $f$ be any Lipschitz function with an arbitrary constant $L$.
It holds that:
\[
\mathbb{E}f(\eta^\rho)-\mathbb{E}f(\eta_{-1})
=\mathbb{E}f\left(w_\lambda\right)-\mathbb{E}f(w_{1})
\leq (1-\lambda)\sqrt{m}L\frac{\mathbb{E}\left|\sum_{q=1}^m\xi_qv_q\right|}{\|v\|_2^2}
\]
As $\rho\rightarrow -1$, the expression $1-\lambda$ vanishes.
We conclude that $g_{m,k}$ is continuous at $-1$.
The function $h_{m,k}:[-1,\infty)\rightarrow\mathbb{R}$ is verified to be continuous at $-1$ via Lebesgue's dominated convergence theorem, which can be applied due to Lemma \ref{lem_youden_dominated}.
The proof is complete.
\end{proof}

\begin{lem}\label{lem_youden_dominated}
Let $m\in\{2,3,\ldots\}$ and $k\in\{1,\ldots,m\}$, as well as $v\in(\mathbb{R}\setminus\{0\})^m$ and $g\in\mathbb{R}$.
There exists a constant $C_{m,v}(g)$ only depending on $m,g,v$, such that for all $z\in D':= \{y\in\mathbb{C}:\Re(y)^2+1\geq \Im(z)^2\}$, 
\[
\left|e^{-\frac{g^2}{2}}\prod_{l=1}^k\Phi\left(\frac{zg|v_l|}{\|v\|_2}\right)\prod_{r=k+1}^m\Phi\left(\frac{-zg|v_r|}{\|v\|_2}\right)\right|
\leq(1+|z|^m) C_{m,v}(g).
\]
Moreover, $\int_{-\infty}^{\infty}C_{m,v}(g)dg\leq \infty$.
\end{lem}

\begin{proof}
We first expand the product:
\[
e^{-\frac{g^2}{2}}\prod_{l=1}^k\left(\frac{1}{2}+\frac{1}{2\pi}\int_0^{\frac{zg|v_l|}{\|v\|_2}}e^{\frac{t^2}{2}}dt\right)\prod_{r=k+1}^m\left(\frac{1}{2}+\frac{1}{2\pi}\int_0^{\frac{-zg|v_r|}{\|v\|_2}}e^{\frac{t^2}{2}}dt\right).
\]
So, up to constant factors and signs, this is a finite sum over terms indexed by subsets $S\subset\{1,\ldots,m\}$ of the form:
\[
e^{-\frac{g^2}{2}}\prod_{j\in S}\int_0^{\frac{zg|v_j|}{\|v\|_2}}e^{\frac{t^2}{2}}dt.
\]
So, the proof is completed if we can show that each of these terms is bounded in absolute value.
Henceforth, fix a nonempty $S\subset\{1,\ldots,m\}$.
First, we perform a change of variable:
\[
\left|e^{-\frac{g^2}{2}}\prod_{j\in S}\int_0^{\frac{zg|v_j|}{\|v\|_2}}e^{\frac{t^2}{2}}dt\right|
\leq 
e^{-g^2/2}\prod_{j\in S}\left(\frac{|zgv_j|}{\|v\|_2}\int_0^{1}\exp\left(-\frac{\Re(z^2) g^2v_j^2t^2}{2\|v\|_2^2}\right)dt\right).
\]
Note that as $z\in D'$, it holds that $\Re(z^2)\geq -1$.
If $\Re(z^2)>0$, then the integral with respect to $t$ is bounded from above by 1. 
Then, as the $|S|$-th absolute moment of a standard Gaussian is finite, we readily see that the expression has an integrable envelope of the form $C_{S,v}(g)|z|^{|S|}\leq C_{m,v}(g)(1+|z|^m)$.
We now suppose that $-1\leq\Re(z^2)<0$.
For such $z$, we can upper bound the term above pointwise by:
\[
e^{-g^2/2}\prod_{j\in S}\left(\frac{|zgv_j|}{\|v\|_2}\int_0^{1}\exp\left(\frac{ g^2v_j^2t^2}{2\|v\|_2^2}\right)dt\right).
\]
We show that this function is integrable with respect to the Lebesgue measure $dg$.
Let $s:=|S|$ and let $d\mathbf{t}$ be the $s$-dimensional Lebesgue measure on $[0,1]^s$.
It holds that:
\[
\int_{-\infty}^{\infty}e^{-g^2/2}\prod_{j\in S}\left(\frac{|zgv_j|}{\|v\|_2}\int_0^{1}\exp\left(\frac{g^2v_j^2t^2}{2\|v\|_2^2}\right)dt\right)\frac{dg}{\sqrt{2\pi}}
\]
\[
=\int_{-\infty}^{\infty}e^{-g^2/2}\int_{[0,1]^{s}}\prod_{j\in S}\left(\frac{|zgv_j|}{\|v\|_2}\exp\left(\frac{g^2v_j^2t_j^2}{2\|v\|_2^2}\right)\right)d\mathbf{t}\frac{dg}{\sqrt{2\pi}}
\]
\[
\leq \int_{[0,1]^{s}}\int_{-\infty}^{\infty}|g|^{s}\exp\left(-\frac{g^2}{2}\left(1-\frac{\sum_{j\in S}v_j^2t_j^2}{\|v\|_2^2}\right)\right)\frac{dg}{\sqrt{2\pi}}d\mathbf{t}
\]
\[
\stackrel{(i)}{\leq }|z|^{s}C_{s}\int_{[0,1]^{s}}\frac{1}{\left(1-\frac{\sum_{j\in S}v_j^2t_j^2}{\|v\|_2^2}\right)^{\frac{s+1}{2}}}d\mathbf{t}.
\]
In $(i)$ we perform a change of variable and then use that the $s$-th absolute moment of a standard Gaussian distribution is finite.
Now we distinguish whether $s=1$ or $s\geq 2$.
If $s=1$, the integral is upper bounded by $1/(1-v_j^2/\|v\|_2^2)$, so the expression is finite.
If instead $s\geq 2$, we can use that $\|v\|_2^2\geq \sum_{j\in S}v_j^2$ to see:
\[
\leq |z|^{s}C_{s}\|v\|_2^{s+1}\int_{[0,1]^{s}}\frac{1}{\left(\sum_{j\in S}v_j^2(1-t_j^2)\right)^{\frac{s+1}{2}}}d\mathbf{t}
\]
\[
\leq |z|^{s}C_{s}\frac{\|v\|_2^{s+1}}{\min_{1\leq k\leq s}\{|v_k|^{s+1}\}}\int_{[0,1]^{s}}\frac{1}{\left(\sum_{j=1}^{s}(1-t_j^2)\right)^{\frac{s+1}{2}}}d\mathbf{t}.
\]
This integral is finite, which follows from the inequality of arithmetic and geometric means:
\[
\int_{[0,1]^{s}}\frac{1}{\left(\sum_{j=1}^{s}(1-t_j^2)\right)^{\frac{s+1}{2}}}d\mathbf{t}
\leq \int_{[0,1]^{s}}\frac{1}{\left(s\prod_{j=1}^{s}(1-t_j^2)^{1/s}\right)^{\frac{s+1}{2}}}d\mathbf{t}
\]
\[
= \frac{1}{s^{\frac{s+1}{2}}}\left(\int_{0}^1\frac{1}{(1-t^2)^{\frac{s+1}{2s}}}dt\right)^{s}
\stackrel{(ii)}{\leq} \frac{1}{s^{\frac{s+1}{2}}}\left(\int_{0}^1\frac{1}{(1-t^2)^{\frac{3}{4}}}dt\right)^{s}<\infty.
\]
In $(ii)$ we use that we are in the case $s\geq 2$.
The proof is complete.
\end{proof}

\section{Additional and technical proofs}
\subsection{Proof of the signal-less imbalanced formula}\label{app_prop_prob_intercept}

\begin{proof}[Proof of Proposition \ref{prop_prob_intercept}]
Let $P:=\spa\{y\}\oplus[0,\infty)^n$.
Arguing as in the proof of Theorem \ref{thm_prob_separation}, we get that the probability of linear separability with intercept is:
\[
S_0(p,n)=
b^n+(1-b)^n+2\sum_{N=1}^{n-1}{n\choose N}b^N(1-b)^{n-N}\sum_{j\geq 1\text{ odd}}\nu_{n-p+j}(P|A_N),
\]
where $\nu_n(P|A_N)=1/{n\choose N}$, $\nu_0(P|A_N)=0$ and for $k\in\{1,\ldots,n-1\}$, defining $r:=k+1-l$, and $l^c:=N-l$, $r^c:=n-N-r$, the $n-k$-th intrinsic volume of $P$ given $A_N$, i.e. $\nu_{n-k}(P|A_N)$ reads:
\[
\sum_{l=1}^{N}{N\choose l}{n-N\choose r}
\mathbb{E}\left[
\Phi_{k}\left(ig\right)^{l}\Phi_k\left(-ig\right)^{r}\right]
\mathbb{E}\left[
\Phi_k\left(g\right)^{l^c}\Phi_k\left(-g\right)^{r^c}\right].
\]
We simplify this expression in the following. 
As we need to take some care with the indexing, we define $\mathbb{I}:=1\{l\leq N\}1\{k+1-l\leq n-N\}$. 
Then:
\[
{n\choose N}{N\choose l}{n-N\choose k+1-l}\mathbb{I}={n \choose k+1}{n-(k+1)\choose N-l}{k+1\choose l}\mathbb{I}.
\]
We introduce some abbreviating notation.
Let $(g,g')\sim\mathcal{N}(0,I_2)$ and define:
\[
\alpha:=b\Phi_{k}\left(ig\right),\quad \tilde{\alpha}:=(1-b)\Phi_{k}\left(-ig\right),\quad \lambda:=b\Phi_k\left(g'\right),\quad \tilde{\lambda}:=(1-b)\Phi_{k}\left(-g'\right).
\]
Finally, note that for any $k\in\{1,\ldots,n-1\}$, we have:
\[
\sum_{N=1}^{n-1}{n\choose N}b^N(1-b)^{n-N}\nu_{n-k}(P|A_N)
\]
\[
=\sum_{N=1}^{n-1}{n\choose N}\sum_{l=1}^{k}{N\choose l}{n-N\choose k+1-l}
\mathbb{E}\left[
\alpha^{l}\tilde{\alpha}^{k+1-l}
\lambda^{N-l}\tilde{\lambda}^{n-N-(k+1-l)}\right]\mathbb{I}
\]
\[
={n\choose k+1}\sum_{N=1}^{n-1}{n-(k+1)\choose N-l}\sum_{l=1}^{k}{k+1\choose l}
\mathbb{E}\left[
\alpha^{l}\tilde{\alpha}^{k+1-l}
\lambda^{N-l}\tilde{\lambda}^{n-N-(k+1-l)}\right]\mathbb{I}
\]
\[
\stackrel{(i)}{=}{n\choose k+1}\sum_{N'=0}^{n-(k+1)}{n-(k+1)\choose N'}\sum_{l=1}^{k}{k+1\choose l}
\mathbb{E}\left[
\alpha^{l}\tilde{\alpha}^{k+1-l}
\lambda^{N'}\tilde{\lambda}^{n-N'-(k+1)}\right]
\]
\[
={n\choose k+1}\mathbb{E}\left[\left(\alpha+\tilde{\alpha}\right)^{k+1}\left(\lambda+\tilde{\lambda}\right)^{n-(k+1)}\right].
\]
In $(i)$, we perform the change of variable $N':=N-l$ and remove the indices outside of the indicator function $\mathbb{I}$.
\end{proof}

\subsection{Proof of the bound for sign-flip noise}\label{app_bounds}

\begin{proof}[Proof of Proposition \ref{prop_bound_sign_flip}]
We prove the case with intercept.
The proof of the case without intercept is analogous, the proof of signal-less case with class imbalance follows the same ideas.

We assume without loss of generality that $\beta^*=e_1$.
Then, $x_{j,2:p}$ is independent of $y_j$.
Moreover, $v_j:=y_jx_{j,1}$ is independent of $y_j$.
To see why, note that:
\[
\mathbb{P}[y_j=1\cap v_j>0]
=\mathbb{P}[y_j=1\cap x_{j,1}>0]
=\mathbb{P}[y_j=1| x_{j,1}>0]\mathbb{P}[x_{j,1}>0]
=\frac{\delta}{2}.
\]
Similarly, we find that the probability that $y=-1$ and $v_j>0$ is $\delta/2$, so $v_j>0$ with probability $\delta$, and independence follows.
To emphasize this independence, we introduce the 1/2-Rademacher random variables $r_j$, independent of $v_j$.
Consequentially, the probability that the data are linearly separable with intercept is equal to:
\[
S_0(n,p)=\mathbb{P}\left[\bigcup_{\beta,\beta_0\in\mathbb{R}^{p+1}}\bigcap_{j=1}^nv_j\beta_j+r_j\beta_0+x_{j,2:p}^T\beta_{2:p}>0\right]=:\mathbb{P}[\text{sep}].
\]
Let $A_N$ be the event that the first $N$ out of $n$ values of $v$ are positive, and the remaining negative.
In the sign-flip-model, conditional on $sign(v_j)$, $|v_j|$ has the same distribution as $|x_{j,1}|$.
We conclude that conditional on $A_N$, the event $\text{sep}$ is invariant under $\delta$.
Consequentially, using \eqref{eq_cover}, we find:
\[
\frac{1}{2^{n-1}}\sum_{k=0}^{p}{n-1\choose k}
=S_0(n,p,1/2)=\sum_{N=0}^n{n\choose N}\frac{1}{2^N}\frac{1}{2^{n-N}}\mathbb{P}\left[\text{sep}|A_N\right]
\]
\[
\geq \frac{1}{(2\delta)^n}\sum_{N=0}^n{n\choose N}\delta^N(1-\delta)^{n-N}\mathbb{P}\left[\text{sep}|A_N\right]
=\frac{S_0(n,p,\delta)}{(2\delta)^n}.
\]
The proof is complete.
\end{proof}

\subsection{Proof of upper bound with statistical dimension}\label{app_bound_dimension}

\begin{proof}[Proof of Theorem \ref{thm_bound_dimension}]
Let $L_{p}$ be a linear subspace of dimension $p$ and $\theta$ be a uniform random rotation in $\mathbb{R}^n$.
Let $A_N$ be the event that the first $N$ entries of $v$ are strictly positive, and the remaining $n-N$ strictly negative. 
For any $N\in\{1,\ldots,n-1\}$, by Lemma \ref{lem_bound_P},
\[
\mathbb{P}[P\cap \theta L_{n-k}\neq \{0\}|A_N]
\]
\[
\leq \mathbb{P}\left[\left(\mathbb{R}^{N}\times[0,\infty)^{n-N}\cup[0,\infty)^{N}\times\mathbb{R}^{n-N}\right)\cap \theta L_{n-k}\neq \{0\}|A_N\right]
\]
\[
\leq\mathbb{P}\left[\mathbb{R}^{N}\times[0,\infty)^{n-N}\cap \theta L_{n-k}\neq \{0\}|A_N\right]
\]
\[
+\mathbb{P}\left[[0,\infty)^{N}\times\mathbb{R}^{n-N}\cap \theta L_{n-k}\neq \{0\}|A_N\right].
\]
We now calculate the intrinsic volumes of $P_N:=\mathbb{R}^N\times[0,\infty)^{n-N}$.
Note that the vector $g\in\mathbb{R}^n$, projected onto $P_N$, cannot land on a face of dimension less than $N$.
It lands on a face of dimension $k\geq N$, if exactly $k-N$ out of $n-N$ indices in $\{N+1,\ldots,n\}$ satisfy $g_j<0$.
So, the $k$-th intrinsic volume of $P_N$ is:
\[
\nu_k(P_N)=\begin{cases}
0,&\text{if }k\leq N,\\
\frac{1}{2^{n-N}}{n-N\choose k-N},&\text{if }k>n.
\end{cases}
\]
Therefore, the statistical dimension of $P_N$ satisfies:
\[
\delta(P_N)=\sum_{k=N}^{n}k\frac{1}{2^{n-N}}{n-N\choose k-N}
=\sum_{k'=0}^{n-N}(k'+N)\frac{1}{2^{n-N}}{n-N\choose k'}=\frac{n+N}{2}.
\]
Similarly,
\[
\delta(P_{n-N})=\frac{n+(n-N)}{2}=n-\frac{N}{2}.
\]
By Theorem \ref{thm_kinematic} and \eqref{eq_tailbounds}, we find that:
\[
\mathbb{P}\left[P_N\cap \theta L_{p-1}\neq \{0\}|A_N\right]
\leq 
\mathbb{P}[\nu(P_N)-\delta(P_N)\geq n-(p-1)-\delta(P_N)|A_N].
\]
Using the approximate kinematic formula \eqref{eq_approx_kin_up}, we find:
\[
\frac{N}{n}\leq 1-\frac{2(p-1)}{n}-4\sqrt{\frac{t}{n}}\quad\Rightarrow\quad \mathbb{P}\left[P_N\cap \theta L_{p-1}\neq \{0\}|A_N\right]\leq \exp\left(-t\right).
\]
Similarly, we find:
\[
\frac{N}{n}\geq \frac{2(p-1)}{n}+4\sqrt{\frac{t}{n}}\quad\Rightarrow\quad \mathbb{P}\left[P_{n-N}\cap \theta L_{p-1}\neq \{0\}|A_N\right]\leq \exp\left(-t\right).
\]
Together, this gives the condition:
\begin{equation}\label{eq_condition_N}
\frac{2(p-1)}{n}+4\sqrt{\frac{t}{n}}\leq \frac{N}{n}\leq 1-\frac{2(p-1)}{n}-4\sqrt{\frac{t}{n}}.
\end{equation}
Since $\mathbb{P}[yx^T\beta^*<0]=:1-\delta\leq \delta$, condition \eqref{eq_condition_delta} implies that:
\begin{equation}\label{eq_condition_delta_ul}
\frac{2(p-1)}{n}+\left(4+\frac{1}{\sqrt{2}}\right)\sqrt{\frac{t}{n}}\leq \delta\leq 1-\frac{2(p-1)}{n}-\left(4+\frac{1}{\sqrt{2}}\right)\sqrt{\frac{t}{n}},
\end{equation}
We define the set of $N$ satisfying \eqref{eq_condition_N} by $N'\subset\{1,\ldots,n-1\}$.
Given \eqref{eq_condition_delta}, the probability that any $A_N$ with $N\in N'$ occurs is bounded as follows:
\[
 \mathbb{P}\left[\bigcup_{N\in N'}A_N\right]
=\mathbb{P}\left[
\frac{2(p-1)}{n}+4\sqrt{\frac{t}{n}}\leq \frac{N}{n}\leq 1-\frac{2(p-1)}{n}-4\sqrt{\frac{t}{n}}
\right]
\]
\[
\stackrel{(i)}{\geq} \mathbb{P}\left[
\delta-\sqrt{\frac{t}{2n}}\leq \frac{N}{n}\leq \delta+\sqrt{\frac{t}{2n}}
\right]\stackrel{(ii)}{\geq} 1-2\exp(-t),
\]
where in $(i)$ we used \eqref{eq_condition_delta_ul} and in $(ii)$ Hoeffding's inequality, and by abuse of notation used $N$ as the random variable giving the number of positive entries of $v$.
Finally,
\[
\mathbb{P}[P\cap \theta L_{p-1}\neq \{0\}]
\leq \mathbb{P}\left[\bigcup_{N\not\in N'}A_N\right]+\sum_{N\in N'}\mathbb{P}[A_N]\mathbb{P}[P\cap \theta L_{p-1}\neq \{0\}|A_N]
\]
\[
\leq 3\exp\left(-t\right).
\]
The proof is complete.
\end{proof}

\subsection{Proof of inequality description}\label{app_description}
Here, we provide the proof of the inequality description of the polyhedral cone $P:=\spa\{v\}\oplus[0,\infty)^n$ in Theorem \ref{thm_description}.

\begin{proof}[Proof of Theorem \ref{thm_description}]
We first introduce some notation.
For all pairs $(l,r)\in\mathcal{I}_N:=\{1,\ldots,N\}\times\{N+1,\ldots,n\}$, we define $a^{lr}:=-e_l/|v_l|-e_{r}/|v_{r}|$, where $e_l$ denotes the $l$-th unit vector.
Moreover, we denote the polyhedral cone on the right-hand side of \eqref{eq_description} by $P'$.
So, our goal is to show that $P=P'$.

We first show the inclusion $P\subset P'$.
Take any $x\in P$.
There exists an $\alpha\in \mathbb{R}$ and an $u\in[0,\infty)^n$, such that $x=\alpha v+u$.
Now, take any $(l,r)\in\mathcal{I}_N$.
Since $a^{lr}\in \ker\{v\}\cap(-\infty,0]^n$,
\[
\langle a^{lr},x\rangle =\langle a^{lr},u\rangle\leq 0
\]
So, any element of $P$ satisfies all inequalities in the description of $P'$.
Consequentially, $P\subset P'$.

Since $P$ is a polyhedron, it can be described by finitely many inequalities.
Any inequality whose hyperplane does not intersect $P$ is redundant.
So, we may only consider inequalities that define faces of $P$.
Consequentially, by Lemma \ref{lem_polar}, $P$ may be described using only inequalities of the form $\langle w,x\rangle\leq 0$, where $w\in\ker\{v\}\cap(-\infty,0]^n$.
We now show that any such inequality is implied by inequalities in the description of $P'$.
If so, $P'\subset P$.

Fix any $w\in \ker\{v\}\cap(-\infty,0]^n$.
We show that $w$ is equal to a linear combination of elements in $\{a^{lr}:(l,r)\in\mathcal{I}_N\}$ with positive coefficients.
If so, the inequality $\langle w,x\rangle\leq 0$ is implied by inequalities in the description of $P'$.
We prove this by constructing a $\beta\in[0,\infty)^{|\mathcal{I}_N|}$, such that for $k\in\{1,\ldots,N\}$, 
\begin{equation}\label{eq_poly_coefficient}
\sum_{(l,r)\in\mathcal{I}_N}\beta_{lr}a^{lr}_k=w_k.
\end{equation}
and for $q\in\{N+1,\ldots,n\}$,
\begin{equation}\label{eq_poly_inequality}
 \sum_{(l,r)\in\mathcal{I}_N}\beta_{lr}a^{lr}_q\geq w_q.
\end{equation}
If this is possible, then:
\[
-\sum_{k=1}^Nw_kv_k
\stackrel{\eqref{eq_poly_coefficient}}{=}
-\sum_{k=1}^N\sum_{(l,r)\in\mathcal{I}_N}\beta_{lr}a^{lr}_kv_k
\stackrel{(i)}{=}
\sum_{q=N+1}^n\sum_{(i,j)\in\mathcal{I}_N}\beta_{lr}a^{lr}_qv_q
\stackrel{\eqref{eq_poly_inequality}}{\leq} \sum_{q=N+1}^nw_qv_q.
\]
The identity $(i)$ holds, since $a^{lr}\in\ker\{v\}$.
Moreover, since $w\in \ker\{v\}$, we have:
\begin{equation}\label{eq_poly_identity}
-\sum_{k=1}^Nw_kv_k
=\sum_{q=N+1}^nw_qv_q.
\end{equation}
Therefore, \eqref{eq_poly_inequality} must be an equality for all $r\in\{N+1,\ldots,n\}$.
In summary, if we can find a vector $\beta\in[0,\infty)^{|\mathcal{I}_N|}$ which satisfies \eqref{eq_poly_coefficient} and \eqref{eq_poly_inequality}, then it is a linear combination of elements in $\{a^{lr}:(l,r)\in\mathcal{I}_N\}$ with positive coefficients, identical to $w$.

\begin{algorithm}[htb]
    \caption{A step in the proof of Theorem \ref{thm_description}. Linear expansion of a normal vector $w$ of a face of $P$, using $a^{lr}$ in the description of $P'$.}\label{alg_1}
    
\begin{algorithmic}
\STATE{     \textbf{Input} $w\in \ker\{v\}\cap(-\infty,0]^n$.}
\STATE{     \textbf{Output} $\beta\in[0,\infty)^{|\mathcal{I}_N|}$, satisfying $\sum_{(i,j)\in\mathcal{I}_N}\beta_{ij}a^{ij}=w$.}
\STATE{ Initialize $\beta\leftarrow 0\in[0,\infty)^{|\mathcal{I}_N|}$.}
	\FOR{$k\in\{1,\ldots,N\}$}
		\STATE{ $q\leftarrow N+1$.}
		\WHILE{$\sum_{(l,r)\in\mathcal{I}_N}\beta_{lr}a_k^{lr}> w_k$}
			\IF{$\sum_{(l,r)\in\mathcal{I}_N}\beta_{lr}a_q^{lr}> w_q$}
				\STATE{$\beta_{kq}\leftarrow \frac{1}{a_k^{kq}}\left(w_k-\sum_{(l,r)\neq (k,q)}\beta_{lr}a_l^{lr}\right)\vee\frac{1}{a_q^{kq}}\left(w_q-\sum_{(l,r)\neq (k,q)}\beta_{lr}a_r^{lr}\right)$ .}
			\ENDIF
			\STATE{$q\leftarrow q+1$.}
		\ENDWHILE
	\ENDFOR
\end{algorithmic}
\end{algorithm}

In Algorithm \ref{alg_1}, we create  such a vector $\beta\in[0,\infty)^{|\mathcal{I}_N|}$.
In the following, we prove the correctness of Algorithm \ref{alg_1}.
Fix any $k\in\{1,\ldots,N\}$.
If $w_k=0$, the while-loop in the $k$-th iteration of the for-loop never runs.
Hence, $\beta_{kq}=0$ for all $q\in\{N+1,\ldots,n\}$, as initialized.
Since these indices are not modified in later iterations either, the $\beta$ which is returned satisfies:
\[
\left(\sum_{(l,r)\in\mathcal{I}_N}\beta_{lr}a^{lr}\right)_k
=\sum_{q={N+1}}^n\beta_{kq}a^{kq}_k=0=w_k.
\]
By definition $w_k\leq 0$, so the remaining case is that $w_k<0$.
If this is true, then the while-loop runs at least once, since the sum in its condition vanishes after initialization.
The updates of $\beta_{kq}$ are such that for all $k\in\{1,\ldots,N\}$,
\begin{equation}\label{eq_alg_k}
\sum_{(l,r)\in\mathcal{I}_N}\beta_{lr}a_k^{lr}\geq w_k,
\end{equation}
and for all $q\in\{N+1,\ldots,n\}$,
\begin{equation}\label{eq_alg_l}
\sum_{(l,r)\in\mathcal{I}_N}\beta_{lr}a_q^{lr}\geq w_q.
\end{equation}
In particular, condition \eqref{eq_poly_inequality} is never violated.
On the other hand, the while-loop iterates through the indices $q\in\{N+1,\ldots,n\}$, until \eqref{eq_alg_k} is an equality, in which case \eqref{eq_poly_coefficient} is satisfied, or until $q=n+1$.
We need to show that the latter only occurs if \eqref{eq_alg_k} is an equality.
If $q=n+1$ is reached, then for all $q\in\{N+1,\ldots,n\}$, we have equality in \eqref{eq_alg_l}.
But then, if \eqref{eq_alg_k} is a strict inequality for some $k\in\{1,\ldots,N\}$, we get: 
\[
\sum_{q=N+1}^n v_q\sum_{(l,r)\in\mathcal{I}_N}\beta_{lr}a^{lr}_q
=
\sum_{q=N+1}^n v_qw_q
\stackrel{\eqref{eq_poly_identity}}{=}
-\sum_{k=1}^Nv_kw_k
>
-\sum_{k=1}^Nv_k\sum_{(l,r)\in\mathcal{I}_N}\beta_{lr}a^{lr}_k.
\]
This contradicts that all $a^{lr}$ lie in the kernel of $v$.

We proved that $P\subset P'$ and that $P'\subset P$.
Next, we show that every inequality in the description of $P$ (respectively $P'$) is facet-defining.
Note that $P$ is full-dimensional since it contains $[0,\infty)^n$.
So, any inequality which is not redundant is facet-defining \cite[Theorem 8.1]{schrijver1998theory}.
To reach a contradiction, we assume that for some $(k,q)\in\mathcal{I}_N$, the inequality $\langle a^{kq},\cdot\rangle\leq 0$ is redundant.
An inequality is redundant if removing it from the description of the polyhedron does not change the polyhedron.
Since all constraints in the inequality description of $P'$ are of the form $\langle a^{lr},\cdot\rangle\leq 0$, it follows that $a^{kq}$ can be written as a positive combination of $a^{lr}$, with $(l,r)\in\mathcal{I}_N\setminus\{(k,l)\}$.
However, all coefficients of such $a^{kq}$ are negative, so any positive combination of such vectors has a strictly negative entry in $\{1,\ldots,n\}\setminus\{k,l\}$.
This contradicts that it coincides with $a^{kq}$.
We showed that every inequality is facet-defining.
The proof is complete.
\end{proof}

\subsection{Proof that Algorithm \ref{alg_2} is correct}\label{app_proof_correct}
In this section, we prove that Algorithm \ref{alg_2} is correct, i.e. that it indeed produces the projection of $x$ onto $P$.
To do so, we verify that all conditions in Proposition \ref{prop_projection} are satisfied.
First, we need some auxiliary results.
We start with the observation that the Algorithm terminates.

\begin{lem}\label{lem_alg2_term}
Algorithm \ref{alg_2} terminates.
\end{lem}

\begin{proof}
We first prove that Algorithm \ref{alg_2} terminates.
If it does not terminate, the while loop is run indefinitely with either of four conditions being violated: $|L|<|\bar{L}|$, $x_{|L|+1}/v_{|L|+1}<\Sigma_{LR}$, $|R|<|\bar{R}|$ or $x_{|R|+1}/v_{|R|+1}>\Sigma_{LR}$.
We first show that $|L|=|\bar{L}|$ implies $x_{|L|+1}/v_{|L|+1}\geq \Sigma_{LR}$.

If $|L|=|\bar{L}|$, then $L=\bar{L}$. 
So, $|L|+1\not\in \bar{L}$, and therefore for all $r\in\{N+1,\ldots,n\}$ we have $x_{|L|+1}/v_{|L|+1}\geq x_{r}/v_{r}$.
Moreover, since the fractions $x_{l}/v_{l}$ are indexed increasingly over $\{1,\ldots,N\}$, it follows that:
\[
\frac{y_{|L|+1}}{v_{|L|+1}}
=\frac{\sum_{k\in L\cup R} \frac{y_{|L|+1}}{v_{|L|+1}}v_k^2}{\sum_{k\in L\cup R}v_k^2}
\geq \frac{\sum_{k\in L\cup R} \frac{y_{k}}{v_{k}}v_k^2}{\sum_{k\in L\cup R}v_k^2}=\Sigma_{LR}.
\]
So, $|L|=|\bar{L}|$ implies $x_{|L|+1}/v_{|L|+1}\geq \Sigma_{LR}$.
Similarly, $|R|=|\bar{R}|$ implies $x_{|R|+1}/v_{|R|+1}\leq \Sigma_{LR}$.

By Lemma \ref{lem_projection_in_P}, the algorithm will terminate as soon as $L=\bar{L}$ and $R=\bar{R}$, since then $y\in P$.
So, it can only run indefinitely with $|L|<|\bar{L}|$ or $|R|<|\bar{R}|$.
Suppose that $|L|<|\bar{L}|$ and suppose that the the algorithm runs indefinitely, since $x_{|L|+1}/v_{|L|+1}\geq \Sigma_{LR}$.
Whether $|R|<|\bar{R}|$ or $|R|=|\bar{R}|$, it must hold that after finitely many iterations of the loop, $x_{|R|+1}/v_{|R|+1}\leq \Sigma_{LR}$.
But then, $y$ satisfies all inequalities in the description of $P$, so the algorithm terminates.
Contradiction.
We conclude that Algorithm \ref{alg_2} terminates.
\end{proof}

For $(l,r)\in\{1,\ldots,N\}\times \{N+1,\ldots,n\}$, we use the notation:
\begin{equation}\label{eq_notation_sigma_lr}
\Sigma_{lr}:=\Sigma_{\{1,\ldots,l\}\{N+1,\ldots,r\}},\quad
\langle x,v\rangle_{lr}:=\sum_{i=1}^lx_iv_i+\sum_{j=N+1}^rx_iv_i.
\end{equation}

\begin{lem}\label{lem_average_monotone}
Let $2\leq l\leq  |\bar{L}|$ and $1\leq r-N\leq |\bar{R}|$.
If $x_{l}/v_{l}<\Sigma_{(l-1)r}$, then:
\[
\frac{x_{l}}{v_{l}}<\Sigma_{lr}<\Sigma_{(l-1)r}.
\]
Let $1\leq l\leq  |\bar{L}|$ and $2\leq r-N\leq |\bar{R}|$.
If $x_{r}/v_{r}>\Sigma_{l(r-1)}$, then:
\[
\frac{x_r}{v_r}>\Sigma_{lr}>\Sigma_{l(r-1)}.
\]
\end{lem}

\begin{proof}
Suppose that $x_{l}/v_{l}<\Sigma_{(l-1)r}$. We show $\Sigma_{lr}<\Sigma_{(l-1)r}$.
It holds that $x_lv_l=(v_l)^2x_l/v_l<(v_l)^2\Sigma_{(l-1)r}$.
Therefore:
\[
\Sigma_{lr}
=\frac{\langle x,v\rangle_{(l-1)r}+v_lx_l}{\langle v,v\rangle_{(l-1)r}+v_l^2}
=\frac{1}{\langle v,v\rangle_{(l-1)r}+v_l^2}
\left(\langle v,v\rangle_{(l-1)r}\frac{\langle x,v\rangle_{(l-1)r}}{\langle v,v\rangle_{(l-1)r}}+v_lx_l\right)
\]
\[
<\frac{1}{\langle v,v\rangle_{(l-1)r}+v_l^2}
\left(\langle v,v\rangle_{(l-1)r}\Sigma_{(l-1)r}+v_l^2\Sigma_{(l-1)r}\right)=\Sigma_{(l-1)r}.
\]
To prove that $\frac{x_{l}}{v_{l}}<\Sigma_{lr}$, we use that $\Sigma_{lr}<\Sigma_{(l-1)r}$:
\[
\Sigma_{lr}\langle v,v\rangle_{lr}
=\langle x,v\rangle_{(l-1)r}+x_lv_l
=\langle v,v\rangle_{(l-1)r}\Sigma_{(l-1)r}+x_lv_l
>\langle v,v\rangle_{(l-1)r}\Sigma_{lr}+x_lv_l.
\]
Rearranging gives $\frac{x_{l}}{v_{l}}<\Sigma_{lr}$.
The proof of the second statement is analogous.
\end{proof}

\begin{lem}\label{lem_projection_increases}
If $y$ is the result of Algorithm \ref{alg_2}, then for all $j\in\{1,\ldots,n\}$, $x_j\leq y_j$.
\end{lem}

\begin{proof}
Fix any $l\in\{1,\ldots,N\}$.
Algorithm \ref{alg_2} terminates by Lemma \ref{lem_alg2_term}.
Let $L,R$ be the index sets at which Algorithm \ref{alg_2} terminates.
If $l\not\in L$, then $x_l=y_l$.
Suppose instead that $l\in L$.
Since the index $|L|$ is included in $L$, there exists an $r\in R$, for which $x_{|L|}/v_{|L|}<\Sigma_{(|L|-1)r}$.
By Lemma \ref{lem_average_monotone} it follows that, $x_{|L|}/v_{|L|}<\Sigma_{|L|r}$ and $\Sigma_{|L|r}< \Sigma_{|L||R|}=\Sigma_{LR}$.
We conclude that:
\[
\frac{x_l-y_l}{v_l}=\frac{x_l}{v_l}-\Sigma_{LR}\leq \frac{x_{|L|}}{v_{|L|}}-\Sigma_{LR}<0.
\]
This concludes the proof for $l\in\{1,\ldots,N\}$, the proof for $r\in\{N+1,\ldots,n\}$ is analogous.
\end{proof}

\begin{lem}\label{lem_projection_kernels}
If $y=y^{LR}(x)$ is of the form \eqref{eq_defn_projection}, then it holds that $\langle x-y,v\rangle =0$ and $\langle x-y,y\rangle =0$.
\end{lem}

\begin{proof}
We prove that $\langle x-y,v\rangle =0$.
Since $x$ and $y$ coincide outside $L\cup R$, it holds that $\langle x-y,v\rangle=\langle x-y,v\rangle_{|L||R|}$.
Note that:
\[
\langle y,v\rangle_{|L||R|}=\Sigma_{LR}\langle v,v\rangle_{|L||R|}=\langle x,v\rangle_{|L||R|}.
\]
It follows that $\langle x-y,v\rangle =0$.
Similarly,
\[
\langle x,y\rangle_{|L||R|}
=\Sigma_{LR}\langle x,v\rangle_{|L||R|}
=\Sigma_{LR}^2\langle v,v\rangle_{|L||R|}
=\langle y,y\rangle_{|L||R|}.
\]
We conclude that $\langle x-y,y\rangle = \langle x-y,y\rangle_{|L||R|}=0$.
The proof is complete.
\end{proof}

\begin{thm}\label{thm_alg_2_correct}
Algorithm \ref{alg_2} terminates and is correct.
\end{thm}

\begin{proof}
By Lemma \ref{lem_alg2_term}, Algorithm \ref{alg_2} terminates.
To prove that $y$ is indeed the projection of $x$ onto $P$, we verify the conditions in Proposition \ref{prop_projection}.
Condition 1. is met, since the algorithm terminates.
By Lemma \ref{lem_polar}, $P^\circ=(-\infty,0]^n\cap \ker\{v\}$.
From Lemma \ref{lem_projection_increases}, we see that $x-y\in (-\infty,0]^n$.
and by Lemma \ref{lem_projection_kernels} $\langle x-y,v\rangle=0$.
Finally, by Lemma \ref{lem_projection_kernels}, $\langle x-y,y\rangle =0$.
We conclude that $y$ is the projection of $x$ onto $P$.
The proof is complete.
\end{proof}

\end{appendix}

\end{document}